\documentclass[final,onefignum,onetabnum]{siamltex1213}
\usepackage{amssymb}
\usepackage{amsmath}
\usepackage{graphicx}

\newtheorem{remark}{Remark}
\newtheorem{problem}{Problem}
\newif\iftechreport 
\techreportfalse     

\usepackage{mathtools}
\usepackage[normalem]{ulem} 

\newcommand{\cH}{{\mathcal{H}}}

\newcommand{\cQ}{{\mathcal{Q}}}

\newcommand{\cS}{{\mathcal{S}}}

\newcommand{\cU}{{\mathcal{U}}}

\newcommand{\RR}{\mathbb{R}}

\newcommand{\dom}{{\mathrm{dom}\,}} 

\newcommand{\prox}{\mathbf{prox}}

\newcommand{\TFBF}{T_{\mathrm{FBF}}}

\DeclareMathOperator*{\Min}{minimize}

\DeclareMathOperator*{\Fix}{Fix}
\DeclareMathOperator*{\zer}{zer}

\DeclareMathOperator*{\gra}{gra}

\newcommand\restr[2]{{
  \left.\kern-\nulldelimiterspace 
  #1 
  \vphantom{\big|} 
  \right|_{#2} 
  }}

\newcommand{\bc}{\begin{center}}
\newcommand{\ec}{\end{center}}

\newcommand{\bdm}{\begin{displaymath}}
\newcommand{\edm}{\end{displaymath}}

\newcommand{\beq}{\begin{equation}}
\newcommand{\eeq}{\end{equation}}

\newcommand{\bfl}{\begin{flushleft}}
\newcommand{\efl}{\end{flushleft}}

\newcommand{\bt}{\begin{tabbing}}
\newcommand{\et}{\end{tabbing}}

\newcommand{\beqn}{\begin{align}}
\newcommand{\eeqn}{\end{align}}

\newcommand{\beqs}{\begin{align*}} 
\newcommand{\eeqs}{\end{align*}}  

\usepackage[lined,boxed,commentsnumbered, ruled,vlined]{algorithm2e}

\newcommand\numberthis{\addtocounter{equation}{1}\tag{\theequation}}

\DeclarePairedDelimiter{\dotp}{\langle}{\rangle}

\newcommand{\TFB}{T_{\mathrm{FB}}}
\usepackage{booktabs}

\newcommand{\Scal}[2]{{\bigg\langle{{#1}\:\bigg |~{#2}}\bigg\rangle}}
\newcommand{\scal}[2]{{\left\langle{{#1}\mid{#2}}\right\rangle}}
\newcommand{\pscal}[2]{\langle\langle{#1}\mid{#2}\rangle\rangle} 
 
\newcommand{\menge}[2]{\big\{{#1}~\big |~{#2}\big\}}

\newcommand{\HH}{\ensuremath{{\mathcal H}}}

\newcommand{\Id}{\ensuremath{\operatorname{Id}}\,}

\newcommand{\RPP}{\ensuremath{\left]0,+\infty\right[}}

\newcommand{\RX}{\ensuremath{\left]-\infty,+\infty\right]}}

\newcommand{\NN}{\ensuremath{\mathbb N}}

\newcommand{\exi}{\ensuremath{\exists\,}}

\newcommand{\weakly}{\ensuremath{\:\rightharpoonup\:}}
\newcommand{\inte}{\ensuremath{\operatorname{int}}}
\newcommand{\infconv}{\ensuremath{\mbox{\small$\,\square\,$}}}

\usepackage{mdframed}
\usepackage{etoolbox}
\usepackage{booktabs}
\usepackage{multirow} 
\usepackage{arydshln} 

\def\cut#1{{}}

\title{Forward-Backward-Half Forward 
Algorithm 
for Solving Monotone Inclusions}

\author{Luis M. Brice\~{n}o-Arias\thanks{
              Department of Mathematics, Universidad T\'ecnica Federico Santa Mar\'ia,
              Santiago, Chile,
              \email{(luis.briceno@usm.cl)}} \and Damek Davis\thanks{
              School of Operations Research and Information Engineering, Cornell University, 
              Ithaca, NY 14850
              \email{(dsd95@cornell.edu)}}}

\begin{document}
\maketitle
\slugger{siopt}{xxxx}{xx}{x}{x--x}

\begin{abstract}
Tseng's algorithm finds a zero of the sum of a maximally monotone operator 
and a monotone continuous operator by evaluating the latter
twice per iteration.
In this paper, we 
modify Tseng's algorithm for finding a zero of the sum of three operators,
where we add a cocoercive operator to the inclusion. Since the sum
of a cocoercive and a monotone-Lipschitz operator is monotone and Lipschitz,
we could use Tseng's method for solving this problem, but implementing both operators 
twice per iteration and 
without taking into advantage the
cocoercivity property of one operator. Instead, in our approach, although the 
{continuous monotone} operator must still be evaluated twice, 
we exploit the cocoercivity of one operator by evaluating it only once per 
iteration. Moreover, when the cocoercive or {continuous-monotone} operators 
are zero
it reduces to Tseng's or forward-backward splittings, respectively, unifying in this way 
both algorithms. 
In addition, we provide a {preconditioned} version of the proposed method 
including
non self-adjoint linear operators in the computation of resolvents and the single-valued 
operators involved. This approach allows us to {also} extend previous variable 
metric 
versions of Tseng's and forward-backward methods and simplify their conditions
on the underlying metrics. We also exploit the case when non self-adjoint 
linear operators are triangular by blocks in the primal-dual product space for solving 
primal-dual composite monotone inclusions, obtaining Gauss-Seidel type algorithms
which generalize several primal-dual methods available in the literature. Finally we
explore {applications to the obstacle problem, Empirical Risk Minimization, 
distributed optimization and nonlinear programming and we illustrate the performance of 
the method via some numerical simulations.}
\end{abstract}



\pagestyle{myheadings}
\thispagestyle{plain}
\markboth{Luis M. Brice\~{n}o-Arias \and Damek Davis}{Forward-Backward-Half Forward 
Algorithm for Solving Monotone Inclusions}

\section{Introduction}

This paper is devoted to the numerical resolution of following problem.
\begin{problem}
	\label{prob:main}
Let $X$ be a nonempty closed convex subset of a real Hilbert space $\cH$, let $A 
: \cH \rightarrow 2^\cH$ {and $B_2 : \cH 
\rightarrow 2^{\cH}$ be maximally monotone operators, with $B_2$ single valued in $\dom 
B_2\supset\dom A\cup X$, and let $B_1 : \cH 
\rightarrow \cH$ be $\beta$-cocoercive\footnote{An operator $C : \cH \rightarrow 
\cH$ is $\beta$-cocoercive for some $\beta > 0$ provided that $\dotp{Cx - Cy, x-y} 
\geq \beta\|Cx - Cy\|^2$.}, for 
some $\beta>0$. Moreover 
assume that $B_2$ is continuous on $\dom A\cup 
X$ 
and that $A+B_2$ is 
maximally monotone. }The 
problem is to 
\begin{equation}
\label{e:main}
	\text{find }\quad x\in X\quad \text{such that }\quad 0\in Ax+B_1x+B_2x,
\end{equation}
under the assumption that the set of solutions to \eqref{e:main} is nonempty.
\end{problem}

The wide variety of applications of Problem~\ref{prob:main} involving optimization 
problems, variational inequalities, 
 partial differential equations, image processing, saddle point problems, 
game theory, among others can be explored in 
{\cite{bauschke2017convex,combetteswajs2005} }and the 
references therein.
As an important application, consider the case of composite optimization problems of 
the form
 \begin{align}\label{eq:primal_before_PD}
 	\Min_{\mathrm{x} \in \mathrm{H}}\, \mathrm{f}(\mathrm{x}) + 
 	\mathrm{g}(\mathrm{L}\mathrm{x}) + \mathrm{h}(\mathrm{x}),
 \end{align}
 where $\mathrm{H}$ and $\mathrm{G}$ are real Hilbert spaces, $\mathrm{L} : 
 \mathrm{H} \rightarrow \mathrm{G}$ is linear and bounded,  
 $\mathrm{f} :\mathrm{H} \rightarrow (-\infty, \infty]$ and $\mathrm{g} : \mathrm{G} 
 \rightarrow (-\infty, \infty]$ are 
 lower semicontinuous, convex, and proper, and  $\mathrm{h}  : \mathrm{H} \rightarrow 
 \RR$ is 
 convex 
 differentiable with $\beta^{-1}$-Lipschitz gradient. Since $g$ may be non smooth, 
 primal algorithms in this context
 need to evaluate $\prox_{\mathrm{g} \circ L}$ or invert $L$ which can be costly 
 numerically. In order to overcome this difficulty, fully split primal-dual algorithms are 
 proposed, e.g., {  in  
 \cite{briceno2011monotone+,condat2013primal,vu2013splitting}}, in 
 which only 
 $\prox_{\mathrm{g}}$, $L$, and $L^\ast$ are computed. These algorithms follow
 from the first order optimality conditions of  \eqref{eq:primal_before_PD}, which, 
under qualification conditions, 
can be written as Problem~\ref{prob:main} with
 \begin{align}\vspace{-10pt}
 	\label{e:primdualops}
 X=\HH=\mathrm{H}\times \mathrm{G},&&A = \partial \mathrm{f} \times 
 \partial \mathrm{g}^\ast, && 
 	B_1 = \nabla \mathrm{h} \times \{0\}, && B_2 = 
 	\begin{bmatrix} 0 & \mathrm{L}^\ast \\ -\mathrm{L} & 0 \end{bmatrix},
 \end{align}
 {where we point out that $B_2$ is monotone and Lipschitz but {\em not 
 cocoercive}, 
 	because it is skew linear and, for every $x\in\HH$, $\scal{x}{B_2x}=0$.}
 We have that, for any solution $x=(\mathrm{x}_1^\ast, \mathrm{x}_2^\ast)\in 
 \zer(A+ B_1 + B_2)$, $\mathrm{x}_1^\ast$ 
 solves~\eqref{eq:primal_before_PD}, where we denote $\zer T=\menge{x\in\HH}{0\in Tx}$
 for any set valued operator $T\colon\HH\to 2^{\HH}$. 
 A method proposed in \cite{vu2013splitting} solves \eqref{eq:primal_before_PD} in a more 
 general context
 by using forward-backward splitting (FB) in the product space with {the} metric
 $\scal{\cdot}{\cdot}_V=\scal{V\cdot}{\cdot}$ for the operators
 $V^{-1}(A+B_2)$ and $V^{-1}B_1$ with a specific choice of self-adjoint strongly 
 monotone linear operator $V$. We recall that the forward-backward splitting 
\cite{combettes2004solving,bruck1975,lions1979splitting,goldstein1964}
 finds a zero of the sum of a maximally monotone and a cocoercive operator, which
 is a particular case of  Problem~\ref{prob:main} when $X=\HH$ and $B_2=0$. This 
 method provides a sequence obtained from the fixed point iteration of the 
 nonexpansive operator (for some 
 $\gamma\in]0,2\beta[$)
 $$
 \TFB := J_{\gamma A}\circ({\Id} - \gamma B_1),
 $$ 
 which converges weakly to a zero of $A+B_1$. Here ${\Id}$ stands for the 
 identity map in 
 $\HH$ and{, for every set valued operator $M\colon\HH\to 2^{\HH}$, 
 $J_{M}=({\Id}+M)^{-1}\colon\HH\to 2^{\HH}$ is the 
 resolvent of $M$, which is single valued and 
 nonexpansive when $M$ is maximally monotone.}  
In the context 
of \eqref{e:primdualops}, the operators $V^{-1}(A+B_2)$ and $V^{-1}B_1$ are 
 maximally monotone and $\beta$-cocoercive in the metric 
 $\scal{\cdot}{\cdot}_V=\scal{V\cdot}{\cdot}$, 
 respectively, which ensures the convergence of the forward-backward splitting. The 
 choice of $V$
 permits the explicit computation of $J_{V^{-1}(A+B_2)}$, which leads to a sequential 
 method that generalizes the algorithm proposed in \cite{chambolle2011first}. A variant 
 for solving 
 \eqref{eq:primal_before_PD} in the case when $h=0$ is proposed in 
 \cite{he2012convergence}. However, previous methods need the skew linear 
 structure of $B_2$ in order to obtain an implementable method.

{ An example in which a non-linear continuous operator $B_2$ arises naturally is
the convex constrained optimization problem
\begin{equation}
	\min_{\substack{x\in C\\g(x)\le 0}}f(x),
\end{equation}
where $f\colon\HH\to\RR$ is convex differentiable with $\beta^{-1}$-Lipschitz-gradient,
$C\subset\HH$ is nonempty, closed and convex, and $g\colon\HH\to\RR$ is a 
$\mathcal{C}^1$ and convex function. The Lagrangian function in this case takes the form
\begin{equation}
L(x,\lambda)=\iota_C(x)+f(x)+\lambda g(x)-\iota_{\RR_+}(\lambda),
\end{equation}
which, under standard qualification conditions can be found by solving the monotone 
inclusion (see \cite{rockafellar1970saddle})
\begin{equation}
0\in A(x,\lambda)+B_1(x,\lambda)+B_2(x,\lambda),
\end{equation}
where $A\colon (x,\lambda)\mapsto N_Cx\times N_{\RR_+}\lambda$ is maximally 
monotone, $B_1\colon (x,\lambda)\mapsto (\nabla f(x),0)$ is cocoercive, and 
$B_2\colon (x,\lambda)\mapsto (\lambda\nabla g(x),-g(x))$ is monotone and continuous
\cite{rockafellar1970saddle}. Of course, the problem can be easily extended to 
consider finitely many
inequality and equality constraints and allow for more general lower semicontinuous 
convex functions than $\iota_C$, but we prefer the simplified version for the ease of 
presentation. Note that the non-linearity of $B_2$ does not allow to use previous 
methods in this context.
}
 
 In the {case when $B_2$ is $L$-Lipschitz for some $L>0$}, since $B:=B_1+B_2$ 
 is monotone 
 and $(\beta^{-1}+L)$--Lipschitz 
 continuous, the forward-backward-forward splitting (FBF) proposed by Tseng in 
 \cite{tseng2000modified} solves Problem~\ref{prob:main}. This method
 generates a sequence from the fixed point iteration of the operator  
 \begin{align*}
 	\TFBF := P_X\circ \left[({\Id} - \gamma B) \circ J_{\gamma A} \circ 
 	({\Id}- 
 	\gamma B) + \gamma B\right],
 \end{align*}
 which converges weakly to a zero of $A+B$, provided that 
 $\gamma\in]0,(\beta^{-1}+L)^{-1}[$. However, this approach has two drawbacks:
 \begin{enumerate}
 	\item  FBF needs to evaluate $B=B_1+B_2$ twice per iteration, without taking into 
 	advantage the cocoercivity property of $B_1$. In the particular case when $B_2=0$,
 	this method computes $B_1$ twice at each iteration, while the forward-backward 
 	splitting needs only one computation of $B_1$ for finding a zero of $A+B_1$. Even if 
 	we cannot ensure that FB is more efficient than FBF in this context, the cost 
 	of each iteration of FB is lower than that of FBF, especially when the computation 
 	cost of $B_1$ is high. This is usually the case, for instance, when $A$, $B_1$, and 
 	$B_2$
 	are as in \eqref{e:primdualops} and we aim at solving \eqref{eq:primal_before_PD}
 	representing a 
 	variational formulation of some partial differential equation (PDE). In this case, 
 the computation of $\nabla\mathrm{h}$ frequently amounts to solving a PDE, which
  is computationally costly.
 	
 	\item The step size $\gamma$ in FBF is bounded above by $(\beta^{-1}+L)^{-1}$, 
 	which in the case when the influence of $B_2$ in the problem is low ($B_2\approx 0$) 
 	leads to a method whose step size cannot go too far beyond $\beta$. In the case 
 	$B_2=0$, the step size $\gamma $ in FB is bounded by $2\beta$. This can affect
 	the performance of the method, since very small stepsizes can lead to slow 
 	algorithms.
 \end{enumerate}
{In the general case when $B_2$ is monotone and continuous, we can also 
apply a version of the method in \cite{tseng2000modified} which uses line search for 
choosing the step-size at each iteration. However, this approach share the disadvantage of 
computing 
twice $B_1$ by iteration and, moreover, in the line search $B_1$ has to be computed several 
times up to find a sufficiently small step-size, which can be computationally costly.}

 In this paper we propose a splitting algorithm for solving Problem~\ref{prob:main} 
 which overcomes previous drawbacks. The method is derived from the 
 fixed point iteration of the operator $T_{\gamma}: \cH 
\rightarrow \cH$, defined by 
\begin{align}\label{eq:ouroperator}
T_\gamma := P_X\circ \left[({\Id}- \gamma B_2) \circ J_{\gamma A} 
\circ ({\Id} - 
\gamma (B_1+B_2)) + \gamma B_2\right],
\end{align}
for some $\gamma\in]0,\chi(\beta,L)[$, where $\chi(\beta,L)\leq\min\{2\beta,L^{-1}\}$ in the 
case when $B_2$ is $L$-Lipschitz.
The algorithm thus obtained implements $B_1$ only once by iteration and it reduces to
FB or FBF when $X=\HH$ and $B_2=0$, or $B_1=0$, respectively, and in these cases 
we have $\chi(\beta,0)=2\beta$ and $\lim_{\beta\to+\infty}\chi(\beta,L)=L^{-1}$. 
{Moreover, 
in the case when $B_2$ is merely continuous, 
the step-size is found
by a line search in which $B_1$ is only computed once at each backtracking step.}
These results 
can be found in Theorem~\ref{t:1} in Section~\ref{sec:2}.
Moreover, 
a generalization of FB for finding a point in $X\cap\zer(A+B_1)$ can be derived 
when $B_2=0$. This can be useful when the solution is known to belong to a closed 
convex set $X$, which is the case, for example, in convex constrained minimization.
The additional projection onto $X$ can improve the performance of the method (see, 
e.g., \cite{BAKS16}). 

 Another contribution of this paper is to include in our method non self-adjoint 
 linear operators in the computation of resolvents and other operators involved. More 
 precisely, in Theorem~\ref{thm:asymmetric_metric} in Section~\ref{sec:asymm}, for 
 an 
 invertible linear operator 
 $P$ (not necesarily self-adjoint) we justify 
 the computation of $P^{-1}(B_1+B_2)$ and 
 $J_{P^{-1}A}$, respectively. 
 In the case when $P$ is self-adjoint and strongly monotone, 
 the properties that $A$, $B_1$ and $B_2$ have with the standard metric
 are preserved by $P^{-1}A$, $P^{-1}B_1$, and $P^{-1}B_2$ in the metric 
 $\scal{\cdot}{\cdot}_P=\scal{P\cdot}{\cdot}$. In this context, variable metric versions 
 of FB and FBF have been developed in 
 \cite{combettes2012variable,vu2013variableFBF}. Of course, a similar generalization 
 can be done for our algorithm, but we go beyond this self-adjoint case and we
 implement  $P^{-1}(B_1+B_2)$ and $J_{P^{-1}A}$, where the linear 
 operator $P$ is strongly monotone but non necesarily self-adjoint. The key for this 
 implementation is the decomposition $P=S+U$, where $U$ is self-adjoint and strongly 
 monotone and $S$ is skew linear. Our implementation follows after coupling $S$ with 
 the monotone and Lipschitz component $B_2$ and using some resolvent identities 
 valid for 
 the metric $\scal{\cdot}{\cdot}_U$. One of the important implications of this issue is 
 the justification of the convergence of some Gauss-Seidel type methods in product 
 spaces, which are deduced from our setting for block triangular linear operators $P$.

Additionally,  we
provide a modification 
of the previous method{ in Theorem~\ref{cor:asymmetricnoinversion}, in which 
linear 
operators 
$P$ may vary among iterations. }In the case when, for every iteration $k\in\NN$, 
$P_k$ is self-adjoint, this feature has also been implemented for FB and FBF in  
\cite{combettes2012variable,vu2013variableFBF}
but with a strong dependence between $P_{k+1}$ and $P_k$ coming from 
the variable metric approach. Instead, in the general case, we modify our method for 
 allowing variable metrics and ensuring convergence under weaker 
conditions. For instance, in the case when $B_2=0$ and $P_k$ is self-adjoint and 
$\rho_k$-strongly monotone for some $\rho_k>0$, 
our condition on our FB variable metric version reduces to 
$(2\beta-\varepsilon)\rho_k>1$ for every 
$k\in\NN$. In the case when $P_k={\Id}/\gamma_k$ this 
condition reduces to $\gamma_k<2\beta-\varepsilon$ which is a standard assumption 
for FB with variable stepsizes. Hence, our condition on operators $(P_k)_{k\in\NN}$
can be interpreted as ``step-size'' bounds.

Moreover, in Section~\ref{sec:5} we use our methods in composite primal-dual 
inclusions, obtaining 
generalizations and new versions of several primal-dual methods 
\cite{chambolle2011first,vu2013variableFBF,patrinos2016asym,combettes2012primal}.
We provide comparisons among methods and new bounds on stepsizes which improve 
several bounds in the literature. Finally, for illustrating the flexibility of the proposed 
methods, 
in Section~\ref{sec:6} we apply them to the obstacle problem in PDE's, to empirical 
risk minimization, to distributed operator splitting schemes and to nonlinear constrained 
optimization. 
In the first example, we take advantage to dropping the extra forward step on $B_1$,
which amounts to reduce the computation of a PDE by iteration. In the 
second example, we use non self-adjoint linear operators in order to obtain a 
Gauss-Seidel 
structure which can be preferable to parallel architectures {for high dimensions.
The third example illustrates how the variable metrics allowed by our proposed algorithm
can be used to develop distributed operator splitting schemes with time-varying 
communication networks. The last example illustrates our backtracking line search 
procedure for nonlinear constrained optimization wherein the underlying operator 
$B_2$ is nonlinear and non Lipschitz. Finally, some numerical examples show the 
performance of the proposed algorithms.
}

\section{Convergence theory}
\label{sec:2}
This section is devoted to study the conditions ensuring the convergence 
of {the iterates generated recursively by} $z^{k+1}=T_{\gamma_k}z^k$ for any 
starting point 
$z^0\in\HH$,
where, for every $\gamma>0$, $T_{\gamma}$ is defined in \eqref{eq:ouroperator}.
We first prove that $T_{\gamma}$ is quasi-nonexpansive for a suitable 
choice of $\gamma$ and satisfies 
$\Fix(T_{\gamma}) = \zer(A+B_1+B_2)\cap X$. Using these results we prove the weak 
convergence of iterates $\{z^k\}_{k\in\NN}$ to a solution to Problem~\ref{prob:main}.

\begin{proposition}[Properties of 
$T_{\gamma}$]\label{prop:Tproperties}
\label{p:prop1}Let $\gamma>0$, assume that hypotheses of Problem~\ref{prob:main} 
hold, {and set $S_{\gamma}:=(\Id - \gamma B_2) \circ J_{\gamma A} \circ (\Id- 
	\gamma (B_1+B_2)) + \gamma B_2$.}
Then, 
\begin{enumerate}

\item \label{prop:Tproperties:item:fixed-points} 
{We have $\zer(A+B_1+B_2)\subset\Fix S_{\gamma}$ and 
$\zer(A+B_1+B_2)\cap X\subset\Fix T_{\gamma}$. Moreover, if $B_2$ is 
$L$-Lipschitz in $\dom B_2$ for some $L>0$ and $\gamma<L^{-1}$ we have 
$
\Fix(S_{\gamma}) = \zer(A+B_1+B_2)
$
and
$
\Fix(T_{\gamma}) = \zer(A+B_1+B_2) \cap X.
$
\item \label{prop:Tproperties:item:quasi0}
For all $z^\ast \in 
\Fix(T_{\gamma})$ and $z\in\dom B_2$, by denoting $x:=J_{\gamma A} (z - \gamma 
(B_1+B_2)z 
)$ we have, for every $\varepsilon>0$,
\begin{align*}
\|T_\gamma z - z^\ast\|^2 &\leq\|z - z^\ast\|^2 - (1-\varepsilon) \|z - x\|^2 + \gamma^2 
\|B_2z - 
B_2x\|^2 \\
&\hspace{.35cm} -\! \frac{\gamma}{\varepsilon}\left(2\beta\varepsilon -\! 
{\gamma}\right)\!\|B_1z - B_1 z^\ast\|^2 
\!-\!\varepsilon\left\|z-x-\frac{\gamma}{\varepsilon}(B_1z-B_1z^*)\right\|^2.
\numberthis\label{eq:fejer0}
\end{align*}}
\item\label{prop:Tproperties:item:quasi} Suppose that
$B_2$ is $L$-Lipschitz {in $\dom A\cup X$ for some $L>0$. For all $z^\ast \in 
\Fix(T_{\gamma})$ and $z\in\dom B_2$, by denoting $x:=J_{\gamma A} (z - \gamma 
(B_1+B_2)z 
)$ we have
\begin{align*}
\|T_\gamma z - z^\ast\|^2 &\leq \|z - z^\ast\|^2 - L^{2}(\chi^2- 
\gamma^2)\|z - x\|^2-\frac{2\beta\gamma}{\chi}\left(\chi 
- \gamma\right)\|B_1z - B_1 z^\ast\|^2\nonumber\\ 
&-\frac{\chi}{2\beta}\left\|z-x-\frac{2\beta\gamma}{\chi}(B_1z-B_1z^*)\right\|^2,
\numberthis\label{eq:fejer}
\end{align*}
}
where
\begin{equation}
\label{e:chi}
\chi:=\frac{4\beta}{1+\sqrt{1+16\beta^2 L^2}}\leq 
\min\{2\beta,L^{-1}\}.
\end{equation}
\end{enumerate}
\end{proposition} 
\begin{proof}
Part~\ref{prop:Tproperties:item:fixed-points}: Let $z^*\in\HH$. We have
{
\begin{align}
	\label{e:auxequiv}
z^*\in \zer(A+B_1+B_2)\quad&\Leftrightarrow\quad  0\in Az^*+B_1z^*+B_2z^*\nonumber\\
&\Leftrightarrow\quad  -\gamma (B_1z^*+B_2z^*)\in \gamma Az^*\nonumber\\
&\Leftrightarrow\quad  z^*=J_{\gamma A}\left(z^*-\gamma 
(B_1z^*+B_2z^*)\right).
\end{align}
}
Then, since $B_2$ is single-valued in $\dom A$, if {$z^*\in \zer(A+B_1+B_2)$}
we have $B_2z^*=B_2J_{\gamma A}(z^*-\gamma 
(B_1z^*+B_2z^*))$ and, hence, {$S_{\gamma}z^*=z^*$ }which yields
{$\zer(A+B_1+B_2)\subset \Fix S_{\gamma}$. Hence, if $z^\ast\in 
\zer(A+B_1+B_2)\cap X$ then $z^\ast\in\Fix P_X$ and $z^\ast\in\Fix S_{\gamma}$,
which yields $z^\ast\in \Fix P_X\circ S_{\gamma}=\Fix T_{\gamma}$.} Conversely, 
if { $B_2$ is $L$-Lipschitz in $\dom B_2$ and $z^*\in \Fix S_{\gamma}$ we 
have }
\begin{equation*}
z^*-J_{\gamma A}(z^*-\gamma(B_1+B_2)z^*)=\gamma\left(B_2z^*-B_2J_{\gamma 
A}(z^*-\gamma(B_1+B_2)z^*) \right),
\end{equation*}
which, from the Lipschitz continuity of $B_2$ yields 
\begin{align*}
\|z^*-J_{\gamma A}(z^*-\gamma(B_1+B_2)z^*)\|&=\gamma\|B_2z^*-B_2J_{\gamma 
A}(z^*-\gamma(B_1+B_2)z^*)\| \\
&\leq \gamma L\|z^*-J_{\gamma 
A}(z^*-\gamma(B_1+B_2)z^*)\|.
\end{align*}
Therefore, if $\gamma<L^{-1}$ we deduce $z^*=J_{\gamma 
A}(z^*-\gamma(B_1+B_2)z^*)$ and from \eqref{e:auxequiv}, we deduce 
{$\zer(A+B_1+B_2)=\Fix S_{\gamma}$. Since $T_{\gamma}=P_XS_{\gamma}$ and 
$P_X$ is 
strictly quasi-nonexpansive, the result follows from 
\cite[Proposition~4.49]{bauschke2017convex}.}

Part~\ref{prop:Tproperties:item:quasi0}: 
Let $z^*\in \Fix T_{\gamma}$, {$z\in\dom B_2$} and define $B := B_1 + B_2$, $y:= 
z - \gamma Bz$,   $x:= 
J_{\gamma A} y$, and $z^+ = T_{\gamma}z$. Note that $(x, y-x) \in 
\gra(\gamma A)$ and, from Part~\ref{prop:Tproperties:item:fixed-points}, $(z^\ast, 
-\gamma Bz^\ast) 
\in \gra(\gamma  A)$. Hence, by the monotonicity of 
$A$ and $B_2$, we have $\dotp{x - z^\ast, x-y -\gamma 
Bz^\ast} \leq 0$ and $\dotp{x - z^\ast, \gamma B_2 z^\ast - \gamma  B_2x}\leq 0$. Thus, 
\begin{align*}
 \dotp{x - z^\ast, x-y - \gamma B_2x} &= \dotp{x - z^\ast, \gamma B_1z^\ast} + \dotp{x - 
 z^\ast, x-y -\gamma Bz^\ast} \\
 &\hspace{2.8cm}+\dotp{x - z^\ast, \gamma B_2z^\ast - \gamma B_2x}  \\
 &\leq \dotp{x - z^\ast, \gamma B_1z^\ast}.
\end{align*}
Therefore, we have
\begin{align}
\label{e:desig1sec2}
2\gamma \dotp{x - z^\ast, B_2z - B_2x} 
&=  2\dotp{ x - z^\ast, 
\gamma B_2z + y - x} + 2 \dotp{ x- z^\ast, x - y - \gamma B_2x} \nonumber\\
&\leq 2 \dotp{ x - z^\ast, \gamma Bz + y - x} + 2\dotp{x - z^\ast, \gamma  B_1z^\ast - 
\gamma B_1z} \nonumber\\
&= 2 \dotp{ x - z^\ast, z - x} + 2\dotp{x - z^\ast, \gamma B_1z^\ast - \gamma B_1z} 
\nonumber\\
&=  \|z - z^\ast\|^2 \!- \!\|x - z^\ast\|^2 \!-\!  \|z - x\|^2 \!+\!  2\dotp{x - z^\ast, \gamma 
B_1z^\ast - \gamma B_1z} .
\end{align}
In addition, by cocoercivity of $B_1$, for all $\varepsilon > 0$, we have
\begin{align}
\label{e:desig2sec2}
2\dotp{x - z^\ast, \gamma B_1z^\ast - \gamma B_1z} 
&= 2\dotp{z - z^\ast, \gamma B_1z^\ast - \gamma B_1z} + 2\dotp{x - z, \gamma B_1z^\ast - 
\gamma B_1z}\nonumber \\
&\leq - 2\gamma \beta\|B_1z - B_1 z^\ast\|^2  + 2\dotp{x - z, \gamma B_1z^\ast - 
\gamma B_1z}\nonumber \\
&= - 2\gamma \beta\|B_1z - B_1 z^\ast\|^2 + \varepsilon\|z-x\|^2 + 
\frac{\gamma^2}{\varepsilon}\|B_1z - B_1 
z^\ast\|^2\nonumber \\
&\hspace{4cm}-\varepsilon\left\|z-x-\frac{\gamma}{\varepsilon}(B_1z-B_1z^*)\right\|^2
\nonumber \\
&= \varepsilon\|z-x\|^2 - \gamma\left(2\beta - 
\frac{\gamma}{\varepsilon}\right)\|B_1z - B_1 
z^\ast\|^2\nonumber \\
&\hspace{4cm}-\varepsilon\left\|z-x-\frac{\gamma}{\varepsilon}(B_1z-B_1z^*)\right\|^2.
\end{align}
Hence, combining \eqref{e:desig1sec2} and \eqref{e:desig2sec2}, it follows from  $z^*\in 
X$, the nonexpansivity of $P_X$, and the Lipschitz property of $B_2$ in $\dom B_2\supset 
X\cup\dom A$ that
{
\begin{align}
\label{e:auxprt3}
\|z^+ - z^\ast\|^2
&\leq \|x - z^\ast + \gamma  B_2z - \gamma B_2x\|^2 \nonumber\\
&= \|x - z^\ast\|^2 + 2\gamma\dotp{x - z^\ast, B_2z - B_2x} + \gamma^2\| B_2z - B_2x\|^2 
\nonumber\\
&\leq  \|z - z^\ast\|^2  -  \|z - x\|^2 + \gamma^2 \|B_2z - B_2x\|^2 \nonumber\\
&\hspace{.35cm}+ \varepsilon\|z\!- x\|^2\! -\! \gamma\left(2\beta -\! 
\frac{\gamma}{\varepsilon}\right)\!\|B_1z - B_1 z^\ast\|^2 
\!-\!\varepsilon\left\|z-x-\frac{\gamma}{\varepsilon}(B_1z-B_1z^*)\right\|^2,
\end{align}
and the result follows.

Part 3: It follows from \eqref{e:auxprt3} and the Lipschitz property on $B_2$ that
\begin{align*}
&\|z^+ - z^\ast\|^2\leq \|z - z^\ast\|^2 - L^2\left(\frac{1- 
\varepsilon}{L^2}-\gamma^2\right)\|z - x\|^2 - 
\frac{\gamma}{\varepsilon}\left(2\beta\varepsilon - 
\gamma\right)\|B_1z - B_1 z^\ast\|^2 \\
&\hspace{.35cm}-\varepsilon\left\|z-x-\frac{\gamma}{\varepsilon}(B_1z-B_1z^*)\right\|^2.
\end{align*}}
In order to obtain the largest interval for $\gamma$ ensuring that 
{the second and third terms on the right of} the above equation are 
negative, we choose 
the value $\varepsilon$ so that $\sqrt{1-\varepsilon}/L = 
2\beta\varepsilon$, which yields 
$\varepsilon=(-1+\sqrt{1+16\beta^2L^2})(8\beta^2L^2)^{-1}$. For this 
choice of $\varepsilon$ we obtain $\chi=\sqrt{1-\varepsilon}/L = 
2\beta\varepsilon$.
\end{proof}

{ In the case when $B_2$ is merely continuous, we need the following result,
	which gives additional information to
	\cite[Lemma~3.3]{tseng2000modified} and allows us to guarantee the convergence of 
	the algorithm under weaker assumptions than \cite[Theorem~3.4]{tseng2000modified}.

	\begin{lemma}
		\label{lem:Ts}
	In the context of Problem~\ref{prob:main}, define, for every $z\in\dom B_2$ and $\gamma 
	>0$,
	\begin{equation}
		\label{e:defxphi} 
		x_{z}\colon \gamma\mapsto J_{\gamma 
		A}(z-\gamma(B_1+B_2)z)\quad \text{and \quad }\varphi_z\colon \gamma\mapsto 
		\frac{\|z-x_z(\gamma)\|}{\gamma}.
	\end{equation}
Then,  the following hold:
	\begin{enumerate}
		\item \label{lem:Tsi}$\varphi_z$ is nonincreasing and
		\begin{equation*}
		(\forall z\in\dom A)\quad \lim_{\gamma\downarrow 
			0^+}\varphi_z(\gamma)=
		\|(A+B_1+B_2)^0(z)\|:=\inf_{w\in(A+B_1+B_2)z}\|w\|.
		\end{equation*}
		\item\label{lem:Tsii} For every $\theta\in]0,1[$ and $z\in\dom B_2$, there exists 
		$\gamma(z)>0$ such 
		that, for every $\gamma\in]0,\gamma(z)]$, 
		\begin{equation}
		\label{e:armijocond}
		\gamma\|B_2z-B_2x_z(\gamma))\|
		\leq\theta\|z-x_z(\gamma)\|.
		\end{equation}
	\end{enumerate}
\end{lemma}
\begin{proof}
Part \ref{lem:Tsi}:	Denote $B:=B_1+B_2$.
If $z\in\zer(A+B)$ then it follows from \eqref{e:auxequiv} that $\varphi_z\equiv 0$ and there 
is nothing to prove. Hence, 
assume $z\in\dom B_2\setminus\zer(A+B)$ which yields $\varphi_z(\gamma)>0$ for every 
$\gamma>0$. From the definition of
$J_{\gamma A}$, we have $(z-x_z(\gamma))/\gamma-Bz\in A(x_z(\gamma))$ for every 
$\gamma>0$ and,
from the monotonicity of $A$, we deduce that, for every strictly positive constants 
$\gamma_1$ and $\gamma_2$ we have
\begin{align}
0&\le \Scal{\frac{z-x_z(\gamma_1)}{\gamma_1}-\frac{z-x_z(\gamma_2)}{\gamma_2}}
{x_z(\gamma_1)-x_z(\gamma_2)}\nonumber\\
&=-\frac{\|z-x_z(\gamma_1)\|^2}{\gamma_1}+\left(\frac{1}{\gamma_1}+\frac{1}{\gamma_2}\right)
\scal{z-x_z(\gamma_1)}{z-x_z(\gamma_2)}-\frac{\|z-x_z(\gamma_2)\|^2}{\gamma_2}.
\end{align}
Therefore
\begin{align}
\label{e:DOng}
\gamma_1\varphi_z(\gamma_1)^2+\gamma_2\varphi_z(\gamma_2)^2&\le 
(\gamma_1+\gamma_2)
\Scal{\frac{z-x_z(\gamma_1)}{\gamma_1}}{\frac{z-x_z(\gamma_2)}{\gamma_2}}\nonumber\\
&\le 
\frac{\gamma_1+\gamma_2}{2}(\varphi_z(\gamma_1)^2+\varphi_z(\gamma_2)^2),
\end{align}
which is equivalent to
$
(\gamma_1-\gamma_2)(\varphi_z(\gamma_1)^2-\varphi_z(\gamma_2)^2)\le 0,
$
and the monotonicity of $\varphi_z$ is obtained. The limit follows from
\cite[Lemma~3.3\&Eq~(3.5)]{tseng2000modified}.

Part \ref{lem:Tsii}: As before, if $z\in\zer(A+B)$ we have $z=x_z(\gamma)$ for every 
$\gamma>0$ and, 
hence, there is nothing to prove. From \ref{lem:Tsi} we have that, for every 
$z\in\dom(A)\setminus 
\zer(A+B)$,
\begin{equation*}
0<\|z-x_z(1)\|\le \lim_{\gamma\downarrow 
	0^+}\frac{\|z-x_z(\gamma)\|}{\gamma}= \|(A+B_1+B_2)^0(z)\|.
\end{equation*}
Therefore, $\lim_{\gamma\downarrow 
	0^+}x_z(\gamma)=z$ and from continuity of $B_2$, $\lim_{\gamma\downarrow 
	0^+}\|B_2z-B_2x_z(\gamma)\|=0$. This ensures the existence of $\gamma(z)>0$
such that, for every $\gamma\in]0,\gamma(z)]$, \eqref{e:armijocond} holds.
\end{proof}
\begin{remark}
	Note that the previous lemma differs from \cite[Lemma~3.3]{tseng2000modified}
	because we provide the additional information $\varphi_z$ nonincreasing. This 
	property is used in \cite{Nghia},
	 proved in \cite{Dong}, and will be 	crucial
	for obtaining the convergence of the algorithm with line search to a solution to 
	Problem~\ref{prob:main} under weaker assumptions. We keep our proof for 
	the sake of completeness and because the inequality \eqref{e:DOng} is slightly stronger 
	than that obtained in \cite{Dong}.
\end{remark}
}

\begin{theorem}[Forward-backward-half forward algorithm]
\label{t:1}
{Under the assumptions of Problem~\ref{prob:main}, 
let $z^0 \in \dom A\cup X$, and consider the sequence $\{z_k\}_{k\in\NN}$ recursively 
defined 
by $z^{k+1} := T_{\gamma_k} z^k$ or, equivalently,
 \begin{equation}
 \label{e:algomain1}
 (\forall k\in\NN)\quad 
 \begin{array}{l}
 \left\lfloor
 \begin{array}{l}
 x^k = J_{\gamma_k A}(z^k - \gamma_k (B_1 + B_2)z^k) \\[2mm]
 z^{k+1} = P_X \big(x^k + \gamma_k B_2z^k - \gamma_k 
 B_2x^k\big),
 \end{array}
 \right.
 \end{array}
 \end{equation}
where $\{\gamma_k\}_{k \in \NN}$ is a sequence of 
stepsizes satisfying one of the following conditions:
\begin{enumerate}
	\item Suppose that $B_2$ is $L$-Lipschitz in $\dom A\cup X$. Then, for every $k\in\NN$, 
	$\gamma_k\in[\eta,\chi-\eta]$, where $\eta\in\left]0,\chi/2\right[$ and 
	$\chi$ is defined in~\eqref{e:chi}.
\item Suppose $X\subset\dom A$ and let $\varepsilon\in\left]0,1\right[$, 
$\sigma\in]0,1[$,
and $\theta\in\left]0,\sqrt{1-\varepsilon}\right[$. Then, for every $k\in\NN$, $\gamma_k$ is 
the largest 
$\gamma\in 
\{2\beta\varepsilon\sigma,2\beta\varepsilon\sigma^2,\cdots\}$ satisfying 
\eqref{e:armijocond} with $z=z^k$, and at least one of the following additional conditions 
holds:
\begin{enumerate}
	\item $\liminf_{k\to\infty}\gamma_k=\delta>0$.
	\item $B_2$ is uniformly continuous in any weakly compact subset of
	$X$.
\end{enumerate}
\end{enumerate}
Then, $\{z_k\}_{k\in\NN}$ converges weakly to a solution 
to Problem~\ref{prob:main}.
}
\end{theorem}
\begin{proof}
{In the case when $B_2$ is $L$-Lipschitz in $\dom A\cup X$, it} follows from 
Proposition~\ref{prop:Tproperties}(\ref{prop:Tproperties:item:quasi}) that the sequence 
$\{z^{k}\}_{k \in 
\NN}$ is Fej{\'e}r monotone with 
respect to $\zer(A+B_1+B_2) \cap X$. Thus, to show that $\{z^k\}_{k \in \NN}$  
converges weakly to a solution to Problem~\ref{prob:main}, we only need to prove that all 
of 
its weak subsequential limits lie in $\zer(A+B_1+B_2) \cap X$ {\cite[Theorem 
5.33]{bauschke2017convex}.} Indeed, 
 it follows from Proposition~\ref{prop:Tproperties} and our hypotheses on the stepsizes 
 that, for every $z^*\in\Fix 
 T_{\gamma}$,
\begin{align}
	\label{e:auxfejer}
	\|z^{k}-z^*\|^2-\|z^{k+1}-z^*\|^2&\ge 
	L^2\eta^2\|z^k-x^k\|^2+\frac{2\beta\eta^2}{\chi}\|B_1z^k-B_1z^*\|^2\nonumber\\
	&\hspace{20pt}+\frac{\chi}{2\beta}
	\left\|z^k-x^k-\frac{2\beta\gamma_k}{\chi}(B_1z^k-B_1z^*)\right\|^2.
\end{align}
Therefore, we deduce from \cite[Lemma~3.1]{combettes2001quasi} that 
\begin{equation}
	\label{e:tozero}
	z^k-x^k\to 0
\end{equation}
when $L>0$ and 
$0<\beta<\infty$\footnote{The case $B_1=0$ ($\beta=+\infty$) 
	has been studied by Tseng in \cite{tseng2000modified}. In the case when $B_2=0$ 
	we 
	can also obtain convergence from 
	Proposition~\ref{prop:Tproperties}, since $L=0$ implies $\chi=2\beta$ and even 
	since the first term in the right hand side of \eqref{e:auxfejer} vanishes, the other two 
	terms yield $z^k-x^k\to0$.}. Now let $z\in\HH$ be the weak limit point of some 
	subsequence of 
$\{z^k\}_{k\in\NN}$. Since $z^k\in X$ for every $k\ge 1$ and $X$ is weakly 
sequentially closed {\cite[Theorem 3.34]{bauschke2017convex}} 
we deduce $z\in X$. Moreover, {by denoting $B:=B_1+B_2$,} it follows from $x^k 
= J_{\gamma_k 
A}(z^k - \gamma_k Bz^k)$ that {$u^k := \gamma_k^{-1}(z^k - x^k) - Bz^k + 
Bx^k\in 
(A+B)x^k$. Then, \eqref{e:tozero}, $\gamma_k\ge \eta>0$ and the Lipschitz continuity of 
$B$ yield $u^k 
\rightarrow 0$.
Now, since $A+B_2$ is maximally monotone and $B_1$ is cocoercive with full 
domain, 
$A+B$ is maximally monotone and its graph is}
closed in  the weak-strong topology in $\HH\times\HH$, which yields 
{$0\in Az+Bz$} and the result follows.

{
In the second case, we deduce from 
Proposition~\ref{prop:Tproperties}
(\ref{prop:Tproperties:item:fixed-points}\&\ref{prop:Tproperties:item:quasi0}) and 
$\gamma_k\leq 2\beta\varepsilon\sigma$ that,
for every $z^*\in\zer(A+B_1+B_2)\cap X$ we have
\begin{align*}
\|z^{k} - z^\ast\|^2- \|z^{k+1} - z^\ast\|^2&\geq  (1-\varepsilon) \|z^k - x^k\|^2  
+\! \frac{\gamma_k}{\varepsilon}\left(2\beta\varepsilon - 
{\gamma_k}\right)\|B_1z^k - B_1 z^\ast\|^2 \\
&\hspace{.5cm}+\varepsilon\left\|z^k-x^k-\frac{\gamma_k}{\varepsilon}(B_1z^k-B_1z^*)\right\|^2
-\gamma_k^2 \|B_2z^k - B_2x^k\|^2\\
&\geq  (1-\varepsilon-\theta^2) \|z^k - x^k\|^2  
+\! 2\beta\varepsilon(1-\sigma)\gamma_k\|B_1z^k - B_1 z^\ast\|^2,
\numberthis\label{eq:fejer1}
\end{align*}
where in the last inequality we use the conditions on $\{\gamma_k\}_{k\in\NN}$,
whose existence is guaranteed by Lemma~\ref{lem:Ts}\eqref{lem:Tsii} because $z^k\in 
X\subset \dom A$. 
Then, we deduce from 
\cite[Lemma~3.1]{combettes2001quasi} that 
$z^k-x^k\to 0$.
Now let $z$ be a weak limit point of a subsequence $\{z^{k}\}_{k\in K}$, with $K\subset 
\NN$.
If $\liminf_{k\to\infty}\gamma_k=\delta>0$, from \eqref{e:armijocond} and $z^k-x^k\to 0$
we have $B_2x^k-B_2z^k\to 0$
and the proof is analogous to the previous case. Finally, for the last case, suppose that 
there exists a subsequence of $\{\gamma_{k}\}_{k\in K}$ (called similarly) satisfying 
$\lim_{k\to\infty,k\in K}\gamma_k=0$. Our choice of $\gamma_k$ yields, for every $k\in K$,
\begin{equation}\label{e:palotro}
\theta\|z^k-J_{\tilde{\gamma}_k A}(z^k-\tilde{\gamma}_k B  
z^k)\|/\tilde{\gamma}_k<\|B_2z^k-B_2J_{\tilde{\gamma}_k A}(z^k-\tilde{\gamma}_k B  
z^k)\|,
\end{equation}
 where $\tilde{\gamma}_k=\gamma_k/\sigma>\gamma_k$ and, from 
 Lemma~\ref{lem:Ts}\eqref{lem:Tsi} we 
 have
 \begin{align}
 \label{e:auxLS}
\sigma\|z^k-J_{\tilde{\gamma}_k A}(z^k-\tilde{\gamma}_k B  
z^k)\|/{\gamma}_k&= \|z^k-J_{\tilde{\gamma}_k A}(z^k-\tilde{\gamma}_k B  
 z^k)\|/\tilde{\gamma}_k\nonumber\\
 &\le  \|z^k-J_{{\gamma}_k A}(z^k-{\gamma}_k B  
 z^k)\|/{\gamma}_k,
 \end{align}
 which, from $z^k-x^k\to 0$, yields 
$$\|z^k-J_{\tilde{\gamma}_k A}(z^k-\tilde{\gamma}_k B  
z^k)\|\le \|z^k-x^k\|/\sigma\to 0$$
as $k\to\infty, k\in K$. Therefore, since $z_k\weakly z$, the sequence $\{\tilde{x}^k\}_{k\in 
K}$ defined by
$$(\forall k\in K)\quad \tilde{x}^k:=J_{\tilde{\gamma}_k A}(z^k-\tilde{\gamma}_k B  
z^k)$$ satisfies $\tilde{x}^k\weakly z$ as $k\to+\infty,k\in K$ and 
\begin{equation}
\tilde{w}^k:=\frac{z^k-\tilde{x}^k}{\tilde{\gamma}_k}+ B \tilde{x}^k- B  
z^k\in (A+B_1+B_2)\tilde{x}^k.
\end{equation}
Hence, since $\{z\}\cup\bigcup_{k\in\NN}[\tilde{x}^k,z^k]$ is a weakly compact 
subset 
of $X$ \cite[Lemma~3.2]{Salzo}, it follows from the uniform continuity of $B_2$ that the 
right hand side of 
\eqref{e:palotro} goes to $0$ and, hence, ${(z^k-\tilde{x}^k)}/{\tilde{\gamma}_k}\to 0$ as 
$k\to\infty,k\in K$. Moreover, since $B_1$ is uniformly continuous, $B=B_1+B_2$
is also locally uniformly continuous and $B \tilde{x}^k- B  
z^k\to 0$, which yields $\tilde{w}^k\to 0$ as $k\to+\infty,k\in K$. The result is obtained as in 
the first case since the graph of $A+B$ is weakly-strongly closed in the product topology.
}
\end{proof}

\begin{remark}
	\begin{enumerate}
		\item In \cite[Theorem~3.4]{tseng2000modified} the local boundedness of $z\mapsto 
		\min_{w\in(A+B)z}\|w\|$ is needed to guarantee the convergence of the method with 
		line 
		search. We drop this assumption by using the monotonicity of $\varphi_z$ in 
		Lemma~\ref{lem:Ts}\eqref{lem:Tsi}, which leads us to the inequality \eqref{e:auxLS}.
		\item Since continuity on compact sets yields uniform continuity, in the finite 
		dimensional setting, the assumption on $B_2$ reduces to the mere continuity on $X$ 
		(see \cite[Remark~3.1(v)]{Salzo}). In this case, we do not need to assume further 
		assumptions than those given in Problem~\ref{prob:main}.
	\end{enumerate}
	
	\end{remark}

\begin{remark}
The maximal monotonicity assumption on $A+B_2$ is satisfied, for 
instance, if $\mathrm{cone}(\dom A-\dom 
B_2)=\overline{\mathrm{span}}(\dom A-\dom B_2)$, where, for any set 
$D\subset\HH$, 
$\mathrm{cone}(D)=\menge{\lambda d}{\lambda\in\RR_+, d\in D}$
and $\overline{\mathrm{span}}(D)$ is the smallest closed linear 
subspace of $\HH$ containing $D$ \cite[Theorem~3.11.11]{Zalinescu}. 
\end{remark}

\begin{remark}
	\label{rem:chi}
	{In the case when $B_2$ is $L$-Lipschitz in $\dom A\cup X$, the} stepsize 
	upper 
	bound $\chi= 
	\chi(\beta,L)$ defined in \eqref{e:chi} 
	depends on the 
	cocoercivity parameter $\beta$ of $B_1$ and the Lipschitz parameter 
	$L$ of $B_2$. In order to fully recover Tseng's splitting algorithm or the 
	forward-backward algorithm in the cases when $B_1$ or $B_2$ are zero, 
	respectively, we study the asymptotic behaviour of $\chi(\beta,L)$ 
	when $L\to 0$ and $\beta\to+\infty$. It is easy to verify that
	\begin{align*}
		\lim_{L\to 0}\chi(\beta,L)=2\beta\quad\text{and}\quad \lim_{\beta\to 
			+\infty}\chi(\beta,L)=\frac{1}{L},
	\end{align*}
	which are exactly the bounds on the stepsizes of forward-backward and 
	Tseng's splittings. 
	{ On the other hand, when $B_2$ is continuous, if we choose 
	$\varepsilon\in\left]0,1\right[$ close to 1, 
	$\{\gamma_k\}_{k\in\NN}$ could be larger since the line search starts from 
	$2\beta\varepsilon\sigma$. However, $\theta<\sqrt{1-\varepsilon}$
	should be close to 0 in this case, and condition \eqref{e:armijocond} is more restrictive 
	and satisfied only for 
small values of $\gamma_k$. Conversely, for small values of $\varepsilon$ we restrict 
the sequence $\{\gamma_k\}_{k\in\NN}$ in a small interval but \eqref{e:armijocond} is more 
easily satisfied. The optimal choice of $\varepsilon$ in order to obtain an optimal sequence 
$\{\gamma_k\}_{k\in\NN}$ depends on the properties of the operators involved.
Note that, in the particular case when $B_2\equiv 0$,  \eqref{e:armijocond} is satisfied for 
$\theta=0$ and we can choose $\varepsilon=1$, recovering forward-backward splitting.
On the other hand, when $B_1\equiv0$, we can take $\varepsilon=0$ and 
$\theta\in\left]0,1\right[$ recovering the Tseng's method with backtracking proposed in 
\cite{tseng2000modified}.}
\end{remark}

\section{Forward-backward-half forward splitting with 
non self-adjoint linear operators}
\label{sec:asymm}
In this section, we introduce modified resolvents $J_{ P^{-1} A}$, which depend on an 
invertible 
linear mapping $P$. In some cases, it is preferable to compute the { modified 
resolvent instead of the standard resolvent $J_{A} = ({\Id}+A)^{-1}$} because 
the former may be easier to compute than the latter or, when $P$ is 
triangular by blocks in a product space, the former may \textit{order the component 
computation} 
of the resolvent, replacing a parallel computation with a 
Gauss-Seidel style sequential computation. However, ${P^{-1}A}$ may 
not be maximally monotone. The following result allows us to use 
some non self-adjoint linear operators in the computation of the resolvent 
by using specific metrics. For simplicity, we assume from here 
that $B_2$ is $L$-Lipschitz in $\dom A\cup X$, for some $L\geq 0$.

\begin{proposition}
	\label{prop:new}
Let $A\colon\HH\to 2^{\HH}$ be a maximally monotone operator, let $P\colon\HH\to\HH$ 
be a linear bounded operator, and let $U := (P + P^\ast)/2$ and $S := (P - 
P^\ast)/2$ be the 
self-adjoint and skew symmetric components 
of $P$, respectively. Assume that there exists $\rho>0$ such that 
\begin{equation}
(\forall x\in\HH)\quad 	\rho\|x\|^2\le \scal{Ux}{x}=:\|x\|^2_U.
\end{equation}
Then, we have
\begin{equation}
	J_{P^{-1}A}=J_{U^{-1}(A+S)}({\Id}+U^{-1}S).
\end{equation}
In particular, $J_{P^{-1}A}\colon\HH\to\HH$ is single valued, everywhere defined and 
satisfies
\begin{equation}
(\forall (x,y)\in\HH^2)\quad \scal{J_{P^{-1}A}x-J_{P^{-1}A}y}{Px-Py}\ge 
\|J_{P^{-1}A}x-J_{P^{-1}A}y\|^2_U
\end{equation}
and, hence, $U^{-1}P^*J_{P^{-1}A}$ is firmly nonexpansive in $(\HH,\scal{\cdot}{\cdot}_U)$,
where $\scal{\cdot}{\cdot}_U\colon (x,y)\mapsto\scal{Ux}{y}$.
\end{proposition}
\begin{proof}
Indeed, since $S$ is monotone and everywhere defined, $A+S$ is maximally monotone in 
$\HH$ \cite[Corollary~25.5]{bauschke2017convex} and, from 
\cite[Lemma~3.7]{combettes2012variable}
we have that $U^{-1}(A+S)$ is maximally monotone in $\HH$ with the metric 
$\scal{\cdot}{\cdot}_U\colon(x,y)\mapsto\scal{x}{Uy}$. Hence, $J_{U^{-1}(A+S)}$ is single 
valued (indeed firmly nonexpansive) and, for every $(x,z) \in \cH^2$, we have
\begin{align*}
x = J_{ U^{-1}(A+S)}(z +  U^{-1} Sz) &\quad \Leftrightarrow\quad z 
+ U^{-1}  S z - x \in  
U^{-1}(A+S )x \\
&\quad\Leftrightarrow \quad (U+S)z- (U+S)x \in 
Ax \\
&\quad\Leftrightarrow \quad x = J_{ P^{-1} A} z.
\end{align*}
Hence, for every 
$(x,y)\in\HH^2$, denoting by $p=J_{P^{-1}A}x=J_{U^{-1}(A+S)}(x+U^{-1} Sx)$
and $q=J_{P^{-1}A}y=J_{U^{-1}(A+S)}(y+U^{-1} Sy)$, the firm 
nonexpansivity of 
$J_{U^{-1}(A+S)}$ in $(\HH,\scal{\cdot}{\cdot}_U)$ yields
\begin{align*}
\scal{p-q}{Px-Py}&=
	\scal{p-q}{U\left(x+U^{-1} 
	Sx-(y+U^{-1} Sy)\right)}\\
&=\scal{p-q}{x+U^{-1} 
	Sx-(y+U^{-1} Sy)}_U\\
&\ge \|p-q\|^2_U,
\end{align*}
and the result follows from $\scal{p-q}{Px-Py}=\scal{U^{-1}P^*(p-q)}{x-y}_U$.
\end{proof}

%

\begin{theorem}[New Metrics and $T_\gamma$]\label{thm:asymmetric_metric}
{Under the hypotheses of Problem~\ref{prob:main} and assuming additionally 
that $B_2$ is $L$-Lipschitz in $\dom A\cup X$}, let $P : \cH \rightarrow 
\cH$ be {a} 
bounded linear 
operator, let $U := (P + P^\ast)/2$ and $S := (P - P^\ast)/2$ be the the 
self-adjoint and skew symmetric components 
of $P$, respectively. Suppose that there exists $\rho>0$ such 
that
\begin{align}
\label{e:metricconditions}
\left(\forall x \in \cH \right) \qquad  \rho\|x\|^2 \leq 
\dotp{Ux, x}\quad\text{and}\quad  
K^2<\rho\left(\rho-\frac{1}{2\beta}\right),
\end{align}
where $K\ge0$ is the Lipschitz constant of $B_2-S$.
Let $z^0 \in{\dom A\cup X}$ and let $\{z^k\}_{k\in\NN}$ be the sequence defined 
by 
the 
following iteration: 
\begin{equation}
\label{eq:compressedalg}
(\forall k\in\NN)\quad 
\begin{array}{l}
\left\lfloor
\begin{array}{l}
x^k = J_{ P^{-1} A}(z^k -  P^{-1}(B_1 + B_2)z^k) \\[2mm]
z^{k+1} = P_X^U(x^k +  U^{-1}(B_2z^k - 
B_2x^k-S(z^k-x^k))),
\end{array}
\right.
\end{array}
\end{equation}
 where $P_X^U$ is the projection operator of $X$ under the inner 
 product $\dotp{\cdot, \cdot}_U$.
 Then $\{z^k \}_{k \in\NN}$ converges weakly to a solution to Problem~\ref{prob:main}.
\end{theorem}
\begin{proof}
Note that, since $U$ is invertible from \eqref{e:metricconditions}, by adding and 
subtracting the skew term $S$, 
Problem~\ref{prob:main} is equivalent to
\begin{align}\label{eq:modifiedmono}
\text{find $x \in X$ such that} \quad 0 \in 
U^{-1}(A+S)x + U^{-1}B_1{x} + 
U^{-1}(B_2-S)x.
\end{align}
Because $S$ and $-S$ are both monotone and Lipschitz, 
  $\mathcal{A}:=U^{-1}(A+S)$ is monotone; 
$\mathcal{B}_1:=U^{-1}B_1$ is $\rho \beta$-cocoercive~\cite[Proposition 
1.5]{davis2014convergenceprimaldual}; 
and $\mathcal{B}_2:=U^{-1} (B_2 - S)$ is 
monotone and $\rho^{-1}K$-Lipschitz {in $\dom A\cup X$} under the inner 
product 
$\dotp{\cdot, 
\cdot}_U=\scal{U\cdot}{\cdot}$, where $K$ 
is the Lipschitz constant of $C:=B_2-S$.\footnote{Note that $K\leq 
	L+\|S\|$, but this constant is not precise when, for instance, 
	$B_2=S$. } For the last assertion 
note that, for every $ x,y\,  {\in \dom A\cup X}$, $$\|\mathcal{B}_2x - 
\mathcal{B}_2y\|^2_{U} = \dotp{U^{-1}(Cx - Cy), Cx - Cy} 
\leq \rho^{-1}K^2\|x-y\|^2\leq 
\rho^{-2}K^2\|x-y\|^2_U.$$
Moreover, the stepsize condition reduces to 
\begin{equation}
\label{e:auxiliary}
\gamma=1<\frac{4\beta\rho}{1+\sqrt{1+16\beta^2
		K^2}}=\frac{-\rho+\sqrt{\rho^2+16\beta^2\rho^2
		K^2}}{4\beta K^2}
\end{equation}
or, equivalently, 
\begin{equation}
\label{e:aux221}
(4\beta
K^2+\rho)^2<{\rho^2+16\beta^2\rho^2 
	K^2}\quad 
\Leftrightarrow\qquad 
2\beta K^2+\rho<2\beta\rho^2,
\end{equation}
which yields the second condition in \eqref{e:metricconditions}.
Therefore, since $\mathcal{A}+\mathcal{B}_2=U^{-1}(A+B_2)$ is maximally monotone 
in $(\cH,\|\cdot\|_U)$, the inclusion~\eqref{eq:modifiedmono} meets the conditions of 
Theorem~\ref{t:1} under this metric. {Therefore, by considering the sequence 
generated 
by 
$z^{k+1}=T_1z^k$ for the 
quasi-nonexpansive operator 
\begin{align}
\label{e:defT1}
T_{1} = P_X^U\circ \left[({\Id} -  
\mathcal{B}_2) \circ J_{ \mathcal{A}} 
\circ \big({\Id} - 
(\mathcal{B}_1+\mathcal{B}_2)\big) +  \mathcal{B}_2\right], 
\end{align}
which, from Proposition~\ref{prop:new} reduces to \eqref{eq:compressedalg},}
we obtain a sequence that weakly converges to a fixed point of 
$T_{1}$, and hence, to a solution of $\zer(A+B_1+B_2)\cap 
X$. 
%
%
%
\end{proof}

\begin{remark}
\begin{enumerate}
\item Note that, in the particular case when $P=\Id/\gamma$, the 
algorithm 
\eqref{eq:compressedalg} 
{reduces to \eqref{e:algomain1} when the stepsizes are constant}. Moreover,
$U=P$, $S=0$, $K=L$, $\rho=1/\gamma$ and the second condition in 
\eqref{e:metricconditions} reduces to $\gamma<\chi$ with $\chi$ 
defined in \eqref{e:chi}. Hence, this assumption can be seen as a 
kind of ``step size'' condition on $P$.

\item \label{rem:4}
As in Remark~\ref{rem:chi}, note that the second condition in 
\eqref{e:metricconditions} depends 
on the cocoercivity parameter $\beta$ and the Lipschitz constant 
$L$. In the 
case when $B_1$ is zero, we can take $\beta\to+\infty$ and this 
condition reduces to $K<\rho$. On the other hand, if $B_2$ is zero
we can take $L= 0$, then $K= \|S\|$ and, hence, the condition  
reduces to $\|S\|^2<\rho(\rho-1/(2\beta))$.
In this way we obtain convergent versions 
of Tseng's splitting 
and forward-backward algorithm with non self-adjoint linear operators by
setting $B_1=0$ or $B_2=0$ in \eqref{eq:compressedalg}, respectively.

\item When $S=0$ {and $B_1=0$ or $B_2=0$, }
from Theorem~\ref{thm:asymmetric_metric} we 
recover {the versions of Tseng's 
forward-backward-forward splitting 
\cite[Theorem~3.1]{vu2013variableFBF} 
or forward-backward \cite[Theorem~4.1]{combettes2012variable},
respectively, when the step-sizes and the 
non-standard metrics involved are constant. }Of course, when 
$S=0$, $U=\Id/\gamma$, and  $\rho=1/\gamma$, we recover 
the classical bound for step-sizes in the standard metric case for each method. 

\item For a particular choice 
of operators and metric, the forward-backward method {
	with non-standard metric} discused before
 has been 
used for solving 
primal-dual composite inclusions and primal-dual optimization problems 
{\cite{condat2013primal,vu2013splitting}}. This approach generalizes,
e.g., the method in \cite{chambolle2011first}.
In Section~\ref{sec:5} we compare the application of our method in the primal-dual 
context 
with  \cite{vu2013splitting} and other methods in the literature.
 
\item In the particular instance when
$B_1=B_2=0$, we need $\|S\|<\rho$ and we obtain from 
\eqref{eq:compressedalg} the 
following version of the proximal point algorithm (we consider $X=\HH$ for simplicity)
\begin{align*}
z^0\in\HH,\quad (\forall k\in\NN)\quad z^{k+1}&=J_{ P^{-1} 
A}z^k+U^{-1}S(J_{ P^{-1} 
A}z^k-z^k) \\
&=(\Id-U^{-1}P)z^k+U^{-1}PJ_{ P^{-1} A}z^k.\numberthis\label{e:ppavm}
\end{align*}
Moreover, in the case when $A=B_2=0$, since $U^{-1}\circ S\circ 
P^{-1}=U^{-1}-P^{-1}$, we 
recover from 
\eqref{eq:compressedalg} the gradient-type method:
\begin{equation}
\label{e:gradvm}
z^0\in\HH,\quad (\forall k\in\NN)\quad z^{k+1}=z^k- 
U^{-1}B_1z^k.
\end{equation}
\item In the particular case when $X=\HH$ and $B_2$ is linear, in 
\cite{patrinos2016asym} a method involving $B_2^*$ is proposed. In 
the case when, $B_2$ is skew linear, i.e., $B_2^*=-B_2$ 
\eqref{e:metricconditions} reduces to this method in the case 
$\alpha_n\equiv 1$ and $S=P$. The methods are different in general.
\end{enumerate}

\end{remark}

\section{Allowing variable $P$ and avoiding inversion of $U$}
\label{sec:4}
In Algorithm~\eqref{eq:compressedalg}, the linear operator $U$ 
must be inverted. In this section, for the special case ${\dom B_2=}X = \cH$, 
we show how to replace this sometimes costly inversion with a 
single multiplication by the map $P$, which, in addition, may 
vary at each iteration. 
This new feature is a consequence of Proposition~\ref{prop:classT} below,
 which allows us to obtain from an operator of the class 
$\mathfrak{T}$ in $(\HH,\|\cdot\|_U)$, another operator of the same class 
in $(\HH,\|\cdot\|)$ preserving the set of fixed points. This change 
to the standard metric allows us to use different linear operators at each 
iteration by avoiding classical restrictive additional assumptions of 
the type $U_{n+1}\preccurlyeq U_n(1+\eta_n)$ with $(\eta_n)_{n\in\NN}$
in $\ell^1_+$.
We recall 
that an operator $\mathcal{S}\colon\HH\to\HH$ belongs to the class 
$\mathfrak{T}$
in $(\HH,\|\cdot\|)$ if and only if $\dom \mathcal{S}=\HH$ and 
$(\forall 
y\in\Fix \mathcal{S})(\forall 
x\in\HH)\quad 
\|x-\mathcal{S}x\|^2\leq \scal{x-\mathcal{S}x}{x-y}$.
\begin{proposition}
\label{prop:classT}
Let $U\colon\HH\to\HH$ {be} a self-adjoint bounded linear operator such 
that, for every $x\in\HH$, $\scal{Ux}{x}\geq\rho\|x\|^2$, for some 
$\rho>0$, let $0<\mu\leq \|U\|^{-1}$, and let 
$\mathcal{S}\colon\HH\to\HH$ 
be 
an 
operator in the class $\mathfrak{T}$ in $(\HH,\|\cdot\|_U)$. Then, the 
operator $\mathcal{Q}=\Id-\mu U(\Id-\mathcal{S})$ belongs to the 
class 
$\mathfrak{T}$ 
in $(\HH,\|\cdot\|)$ and $\Fix \mathcal{S}=\Fix \mathcal{Q}$.
\end{proposition}
\begin{proof}
First note that, under the assumptions on $U$ it is invertible and, 
from \cite[Lemma~2.1]{combettes2013variable}, we deduce
\begin{equation}
\label{e:ineqU}
(\forall 
x\in\HH)\quad \|x\|_U^2=\scal{Ux}{x}=\scal{Ux}{U^{-1}Ux}
\geq \|U\|^{-1}\|Ux\|^2,
\end{equation}
and $\Fix \mathcal{S}=\Fix \mathcal{Q}$ thus follows from the 
definition of $\mathcal{Q}$.
Now let $y\in\Fix 
\mathcal{S}$ and $x\in\HH$. We have from 
\eqref{e:ineqU} that
\begin{align}
\|x-\mathcal{S}x\|_U^2\leq \scal{x-\mathcal{S}x}{x-y}_U\:
&\Leftrightarrow\: \|x-\mathcal{S}x\|_U^2\leq 
\scal{U(x-\mathcal{S}x)}{x-y}\nonumber\\
&\Rightarrow\:  \|U\|^{-1}\|U(x-\mathcal{S}x)\|^2\leq 
\scal{U(x-\mathcal{S}x)}{x-y}\nonumber\\
&\Leftrightarrow\: \frac{\|U\|^{-1}}{\mu}\|\mu 
U(x-\mathcal{S}x)\|^2\leq 
\scal{\mu U(x-\mathcal{S}x)}{x-y}\nonumber\\
&\Leftrightarrow\: \frac{\|U\|^{-1}}{\mu}\|x-\mathcal{Q}x\|^2\leq 
\scal{x-\mathcal{Q}x}{x-y}
\end{align}
and, hence, if $\mu\in]0,\|U\|^{-1}]$ we deduce the result.
\end{proof}

\begin{theorem}\label{cor:asymmetricnoinversion}
{Under the hypotheses of Problem~\ref{prob:main}and assuming additionally 
	that $B_2$ is $L$-Lipschitz in $\dom B_2=\HH$}, let $\{P_k\}_{k \in 
\NN}$ be a 
sequence of bounded, 
linear maps from $\cH$ to $\cH$. For each $k \in \NN$, let 
$U_k := (P_k + P_k^\ast)/2$ and ${S}_k := (P_k - 
P_k^\ast)/2$ 
be the self-adjoint and skew symmetric components 
of $P_k$, respectively. Suppose that 
$M:=\sup_{k\in\NN}\|U_k\|<\infty$ and that there 
exist $\varepsilon\in]0,(2M)^{-1}[$, $\rho > 0$, and $\{\rho_k\}_{k 
\in \NN} 
\subseteq  [\rho, 
\infty[$ 
such that, for every $k\in\NN$,
\begin{align}
\label{e:metricconditions2}
\left(\forall x \in \cH \right) \qquad  \rho_k\|x\|^2 \leq 
\dotp{U_kx, x} && \text{and} &&  
K_k^2\leq\frac{\rho_k}{1+\varepsilon}\left(\frac{\rho_k}
{1+\varepsilon}-\frac{1}{2\beta}\right),
\end{align}
where $K_k\ge0$ is the Lipschitz constant of $B_2-S_k$.
Let $\{\lambda_k\}_{k \in \NN}$ be a sequence in $[\varepsilon, 
\|U_k\|^{-1} - \varepsilon]$, let $z^0 \in\HH$, and let $\{z^k\}_{k 
\in \NN}$ be a sequence of points defined by the following iteration: 

\begin{equation}
\label{eq:FBF-asymmetric-no-U}
(\forall k\in\NN)\quad 
\begin{array}{l}
\left\lfloor
\begin{array}{l}
x^k = J_{P_k^{-1} A}(z^k - P_k^{-1}(B_1 + B_2)z^k) \\[2mm]
z^{k+1} = z^k + \lambda_k\left( P_k(x^k - z^k) + B_2z^k - 
B_2x^k\right). 
\end{array}
\right.
\end{array}
\end{equation}
 Then $\{z^k \}_{k \in\NN}$ converges weakly to a solution to Problem~\ref{prob:main}.
\end{theorem}
\begin{proof}
For every invertible and bounded linear map $P : \cH \rightarrow 
\cH$, let us denote by
$\mathcal{T}_{ P} : \cH \rightarrow \cH$ the 
forward-backward-forward 
operator of Theorem~\ref{thm:asymmetric_metric} in the case 
$X=\cH$, which associates, to every $z\in\HH$,
\begin{align*}
\mathcal{T}_{P}z &= x_z + U^{-1}( B_2z - B_2x_z-S( z-x_z)),
\end{align*}
where $x_z = J_{P^{-1} A}(z -P^{-1}(B_1 + B_2)z)$. Recall that, from \eqref{eq:fejer} 
and the proof of Theorem~\ref{thm:asymmetric_metric}, $\mathcal{T}_P$ 
is
a quasi-nonexpansive mapping in $\HH$ endowed with the scalar product 
$\scal{\cdot}{\cdot}_U$.
Observe that multiplying ${\Id} - \mathcal{T}_{P}$ by $U$ on the left 
yields a  $U^{-1}$-free 
expression:
\begin{align*}
({\Id} - \mathcal{T}_{P})(z) &= (z - x_z) +U^{-1}S(z-x_z) -U^{-1}(B_2z -  
B_2x_z)\\
\Leftrightarrow\qquad U({\Id} - \mathcal{T}_{P})(z) &= (U+S)(z - x_z) +  
B_2x_z - B_2z\\
&= P(z - x_z) +  B_2x_z - 
B_2z\numberthis\label{eq:Snotation}.
\end{align*}
Note that, since $\mathcal{T}_P$ is quasi-nonexpansive in 
$(\HH,\|\cdot\|_U)$, 
it follows from 
\cite[Proposition~2.2]{combettes2001quasi} that 
$\mathcal{S}:=(\Id+\mathcal{T}_P)/2$ 
belongs to the class $\mathfrak{T}$ in $(\HH,\|\cdot\|_U)$ and, from 
Proposition~\ref{prop:classT} and \eqref{eq:Snotation} we obtain that the operator
\begin{equation}
\label{e:QP}
	\mathcal{Q}_{P}:={\Id}-\|U\|^{-1}U({\Id}-\mathcal{S})=
	{\Id}-\frac{\|U\|^{-1}}{2}U({\Id}-\mathcal{T}_P)
\end{equation}
belongs to the class $\mathfrak{T}$ in 
$(\HH,\|\cdot\|)$ and $\Fix \mathcal{S}=\Fix 
\mathcal{Q}_{P}=\zer(U({\Id} - 
\mathcal{T}_{P}))=\Fix(\mathcal{T}_{P})=\zer(A+B_1+B_2)$.
Hence, from \eqref{eq:Snotation} and \eqref{e:QP}, the algorithm 
\eqref{eq:FBF-asymmetric-no-U} can be written equivalently as
\begin{align}
z^{k+1}&=z^k-\lambda_k(P_k(z^k - x_{z^k}) +  B_2x_{z^k} - 
B_2z^k)\nonumber\\
&=z^k+2\lambda_k\|U_k\|(\mathcal{Q}_{P_k}z^k-z^k).
\end{align}
Hence, since $0<\liminf\lambda_k\|U_k\|\le \limsup \lambda_k\|U_k\|<1$, 
it follows from 
\cite[Proposition~4.2 and 
Theorem~4.3]{combettes2001quasi} 
that
$(\|z^k-\mathcal{Q}_{P_k}z^k\|^2)_{k\in\NN}$ is a summable sequence
 and
$\{z^k\}_{k\in\NN}$ converges weakly in $(\HH,\scal{\cdot}{\cdot})$ to 
a solution to 
$\cap_{k\in\NN}\Fix \mathcal{T}_{P_k}=\zer(A+B_1+B_2)$ if and only if 
every 
weak limit of the sequence is a solution. Note that, since \eqref{e:metricconditions2} yields
$\|U_k^{-1}\|\leq \rho_k^{-1}$, we have
\begin{align}
\label{e:tozero1}
\|z^k-\mathcal{T}_{P_k}z^k\|_{U_k}^2&=\scal{U_{k}(z^k-\mathcal{T}_{P_k}z^k)}{z^k-
\mathcal{T}_{P_k}z^k}\nonumber\\
&\le \|U_{k}(z^k-\mathcal{T}_{P_k}z^k)\|\,\|z^k-\mathcal{T}_{P_k}z^k\|\nonumber\\
&=\|U_k^{-1}\|\|U_{k}(z^k-\mathcal{T}_{P_k}z^k)\|^2\nonumber\\
&\le 4\|U_k\|^2\rho_k^{-1}\,\|z^k-\mathcal{Q}_{P_k}z^k\|^2\nonumber\\
&\le 4M^2\rho^{-1}\|z^k-\mathcal{Q}_{P_k}z^k\|^2\to 0.
\end{align}
Moreover, since $\mathcal{T}_{P_k}$ coincides 
with $T_1$ defined in \eqref{e:defT1} involving the operators 
$\mathcal{A}_k:=U_k^{-1}(A+S_k)$, $\mathcal{B}_{1,k}=U_k^{-1}B_1$, 
and 
$\mathcal{B}_{2,k}=U_k^{-1}(B_2-S_k)$ which are 
monotone, $\rho_k\beta$-cocoercive, and monotone and 
$\rho_k^{-1}K_k$-lipschitzian in $(\HH,\|\cdot\|_{U_k})$, 
respectively, 
we deduce from \eqref{eq:fejer} 
that, for every 
$z^{\ast}\in\zer(A+B_1+B_2)=\cap_{k\in\NN}\zer(\mathcal{A}_k+
\mathcal{B}_{1,k}+\mathcal{B}_{2,k})$ we have
\begin{align}
\label{e:aux331}
&\rho_k^{-2}K_k^{2}(\chi_k^2- 
1)\|z^k - J_{P_k^{-1} A}(z^k - P_k^{-1}(B_1+ 
B_2)z^k)\|_{U_k}^2 \nonumber\\
&+ \frac{2\beta\rho_k}{\chi_k}\left(\chi_k 
- 1\right)\|U_k^{-1}(B_1z^k - B_1 z^\ast)\|_{U_k}^2\nonumber \\
&+\frac{\chi_k}{2\beta\rho_k}\left\|z^k-J_{P_k^{-1}A} (z^k - P_k^{-1} 
(B_1z^k+ B_2z^k))-
\frac{2\beta\rho_k}{\chi_k}U_k^{-1}(B_1z^k-B_1z^*)\right\|_{U_k}^2\nonumber\\
&\leq \|z^k - z^\ast\|^2_{U_k} - \|\mathcal{T}_{P_k}z^k - 
z^\ast\|^2_{U_k}\nonumber\\
&= -\|\mathcal{T}_{P_k}z^k - z^k\|_{U_k}^2 - 
2\dotp{\mathcal{T}_{P_k}z^k - z^k, 
z^\ast - 
z^k}_{U_k}\nonumber\\
&\leq-\|\mathcal{T}_{P_k}z^k - z^k\|_{U_k}^2 + 
2M\|\mathcal{T}_{P_k}z^k - z^k\|_{U_k}
\|z^\ast - z^k\|,
\end{align}
where 
\begin{equation}
\label{e:chikbound}
\chi_k:=\frac{4\beta 
\rho_k}{1+\sqrt{1+16\beta^2K_k^2}}\leq\rho_k\min\{2\beta,K_k^{-1}\}.
\end{equation}
By straightforward computations in the line 
of \eqref{e:auxiliary} and \eqref{e:aux221} we deduce that 
\eqref{e:metricconditions2} 
implies, for all $k\in\NN$, $\chi_k\geq1+\varepsilon$, $K_k\le \rho_k\le \|U_k\|\le M$ and, 
hence, 
we deduce from \eqref{e:aux331} and \eqref{e:metricconditions2} that 
\begin{multline}
\label{e:tozero2}
\frac{\varepsilon\rho K_k^2}{M^2}\|z^k - J_{P_k^{-1} A}(z^k - 
P_k^{-1}(B_1+ 
B_2)z^k)\|^2+{\varepsilon\rho}\|U_k^{-1}(B_1z^k - B_1 
z^\ast)\|^2\\
+\frac{\rho}{2\beta M}\left\|z^k-J_{P_k^{-1}A} (z^k - P_k^{-1} 
(B_1z^k+ B_2z^k))-
\frac{2\beta\rho_k}{\chi_k}U_k^{-1}(B_1z^k-B_1z^*)\right\|^2\\
\leq-\|\mathcal{T}_{P_k}z^k - z^k\|_{U_k}^2 + 
2M\|\mathcal{T}_{P_k}z^k - z^k\|_{U_k}
\|z^\ast - z^k\|.
\end{multline}
Now, let $z$ be a weak limit of some subsequence of $(z^{k})_{k\in\NN}$ called 
similarly for simplicity. 
We have that $(\|z^\ast - 
z^k\|)_{k\in\NN}$ is bounded and, since \eqref{e:tozero1} implies
$\|z^k-\mathcal{T}_{P_k}z^k\|_{U_k}^2\to0$ we deduce from 
\eqref{e:tozero2}
that{, by denoting  $x^k:=J_{P_k^{-1} A}(z^k - P_k^{-1}(B_1+ 
B_2)z^k)$, that $z^k-x^k\to0$. Hence, since, for every $x\in\HH$,
\begin{equation}
	\|S_kx\|\leq \|(S_k-B_2)x-(S_k-B_2)0\|+\|B_2x-B_20\|\leq 
	(K_k+L)\|x\|\leq\left(M+L\right)\|x\|,
\end{equation} 
we have 
\begin{align}
\label{e:goingto0}
\|P_k(z^k-x^k)\|&=\|(U_k+S_k)(z^k-x^k)\|\nonumber\\
&\leq \|U_k(z^k-x^k)\|+\|S_k(z^k-x^k)\|\nonumber\\
&\leq 
(2M+L)\|z^k-x^k\|\to 0.
\end{align}
Finally, denoting by $B:=B_1+B_2$
we have
\begin{equation}
u^k:=P_k(z^k-x^k)-(Bz^k-Bx^k)\in (A+B)x^k,
\end{equation}
and since $z^k-x^k\to0$ and $B$ is continuous, it follows from 
\eqref{e:goingto0} that $u^k\to 0$ and the result follows from the 
weak-strong closedness of the maximally monotone operator $A+B$ 
and \cite[Theorem 
5.33]{bauschke2017convex}. }
\end{proof}

\newpage7
\begin{remark}
\begin{enumerate}
\item Note that, in the particular case when  $S_k\equiv 0$ and 
$P_k=U_k=\gamma_k^{-1}V_k^{-1}$, we have from 
\cite[Lemma~2.1]{combettes2013variable} that 
$\rho_k=\gamma_k^{-1}\|V_k^{-1}\|$, the conditions on the constants
involved in Theorem~\ref{cor:asymmetricnoinversion} reduce to 
\begin{equation}
\label{e:condit_sym_case}
\frac{\|V_k^{-1}\|}{M}\leq\gamma_k\leq 
\frac{\|V_k^{-1}\|}{\rho},\quad L^2\leq 
\frac{\gamma_k^{-1}\|V_k^{-1}\|}{1+\varepsilon}
\left(\frac{\gamma_k^{-1}\|V_k^{-1}\|}{1+\varepsilon}-\frac{1}{2\beta}\right),
\end{equation}
 for some $0<\rho<M$, for every $k\in\NN$, and \eqref{eq:FBF-asymmetric-no-U} 
 reduces to 
\begin{equation}
	\label{eq:FBF-symmetric-no-U}
(\forall k\in\NN)\quad 
\begin{array}{l}
\left\lfloor
\begin{array}{l}
x^k = J_{\gamma_kV_k A}(z^k - \gamma_k V_k(B_1 + B_2)z^k) \\[2mm]
z^{k+1} = z^k + \frac{\lambda_k}{\gamma_k}\left( V_k^{-1}(x^k - z^k) 
+ \gamma_kB_2z^k - \gamma_kB_2x^k\right).
\end{array}
\right.
\end{array}
\end{equation}
If in addition we assume that $B_2=0$ and, hence $L=0$, 
\eqref{e:condit_sym_case} reduces to $\gamma_k\leq 
\|V_k^{-1}\|2\beta/(1+\varepsilon)$ which is more general than the 
condition in \cite{combettes2012variable} and, moreover, 
we do not need any {compatibility assumption on} $(V_k)_{k\in\NN}$ for 
achieving 
convergence. Similarly, if $B_1=0$, and hence, we can take 
$\beta\to\infty$,  \eqref{e:condit_sym_case} reduces to $\gamma_k\leq 
\|V_k^{-1}\|/(L(1+\varepsilon))$ which is more general than the condition in 
\cite{vu2013variableFBF} and no additional assumption on 
$(V_k)_{k\in\NN}$ is needed. However, \eqref{eq:FBF-symmetric-no-U} 
involves an additional computation of $V_k^{-1}$ in the last step of 
each iteration $k\in\NN$.

\item In the particular case when, for every $k\in\NN$, 
$P_k=U_k=\Id/\gamma_k$, where $(\gamma_k)_{k\in\NN}$ is a real 
sequence, we have $S_k\equiv0$, $K_k\equiv L$, 
$\|U_k\|=\rho_k=1/\gamma_k$, and conditions 
$\sup_{k\in\NN}\|U_k\|<\infty$ and \eqref{e:metricconditions2} reduce 
to
\begin{equation}
0<\inf_{k\in\NN}\gamma_k\leq\sup_{k\in\NN}\gamma_k<\chi, 
\end{equation}
where $\chi$ is defined in \eqref{e:chi} and 
\eqref{eq:FBF-asymmetric-no-U} reduces to
\begin{equation*}
	(\forall k\in\NN)\quad 
	\begin{array}{l}
		\left\lfloor
		\begin{array}{l}
			x^k =J_{\gamma_k A}(z^k - \gamma_k(B_1 + B_2)z^k) \\[2mm]
			z^{k+1} = z^k + \eta_k\left( x^k + 
			\gamma_kB_2z^k - \gamma_kB_2x^k- z^k\right), 
		\end{array}
		\right.
	\end{array}
\end{equation*}
where $\eta_k\in[\varepsilon,1-\varepsilon]$, which is a relaxed version of 
Theorem~\ref{t:1}.

\item As in Remark~\ref{rem:4}, by setting $B_1=0$ or $B_2=0$, we can 
derive from \eqref{eq:FBF-asymmetric-no-U} versions of Tseng's 
splitting and forward-backward algorithm with non self-adjoint linear operators
but without needing the inversion of $U$. In particular, the proximal 
point algorithm in 
\eqref{e:ppavm} reduces to
\begin{equation}
\label{e:ppavmwinv}
z^0\in\HH,\quad (\forall k\in\NN)\quad z^{k+1}=z^k+\lambda 
P(J_{P^{-1} A}z^k-z^k)
\end{equation}
for $\lambda<\|U\|^{-1}$ and, in the case of \eqref{e:gradvm}, 
to avoid inversion is to 
come back to the gradient-type method with the standard metric.

\end{enumerate}
\end{remark}

%


\section{Primal-dual composite 
monotone inclusions with non self-adjoint linear operators}
\label{sec:5}
In this section, we apply our algorithm to composite primal-dual monotone 
inclusions involving a cocoercive and a lipschitzian monotone 
operator. 
\begin{problem}
\label{prob:mi}
Let ${\mathrm H}$ be a real Hilbert space, let ${\mathrm 
X}\subset {\mathrm H}$ be closed and convex, let $z\in{\mathrm 
H}$,  let ${\mathrm A}\colon {\mathrm H} 
\to 
2^{\mathrm H}$ be maximally monotone, let
${\mathrm C}_1 \colon  {\mathrm H} \to{\mathrm H}$ be 
$\mu$-cocoercive, for some $\mu\in\RPP$, and let ${\mathrm C}_2 
\colon  {\mathrm H} \to{\mathrm H}$ be a monotone and
$\delta$-lipschitzian operator, for some $\delta\in\RPP$. Let $m\geq 1$ be an 
integer,  
and, for 
every 
$i\in\{1,\ldots,m\}$, let ${\mathrm G}_i$ be a real Hilbert 
space, let $r_i\in{\mathrm G}_i$, let ${\mathrm B}_i\colon 
{\mathrm G}_i \to 2^{{\mathrm G}_i}$ be maximally monotone, 
let ${\mathrm D}_i\colon {\mathrm G}_i \to 2^{{\mathrm G}_i}$ be 
maximally monotone and $\nu_i$-strongly monotone, for some 
$\nu_i\in\RPP$, and suppose that ${\mathrm L}_i\colon {\mathrm 
H} \to {\mathrm G}_i$ is a nonzero linear bounded operator. The 
problem is to solve the primal inclusion. 
\begin{equation}
\label{e:miprimal}
\text{find }\quad {\mathrm x}\in {\mathrm X}\quad \text{such that 
}\quad {\mathrm z}\in 
{\mathrm A}{\mathrm x}+\sum_{i=1}^m{\mathrm 
L}_i^*({\mathrm B}_i\infconv {\mathrm D}_i)(\mathrm 
{L}_i\mathrm 
{x}-\mathrm 
{r}_i)+\mathrm{C}_1\mathrm{x}+\mathrm{C}_2\mathrm{x}
\end{equation}
together with the dual inclusion
\begin{align*}
&\text{find }\quad \mathrm{v}_1\in  
\mathrm{G}_1,\ldots,\mathrm{v}_m\in  
\mathrm{G}_m\quad \\
&\text{such that 
}\quad 
(\exists \mathrm{x\in 
X})\:
\begin{cases}
\mathrm{z}-\sum_{i=1}^m\mathrm{L}_i^*\mathrm{v}_i\in 
\mathrm{Ax+C}_1\mathrm{x}+\mathrm{C}_2\mathrm{x}\\
(\forall i\in\{1,\ldots,m\})\:\:\mathrm{v}_i\in({\mathrm 
B}_i\infconv {\mathrm D}_i)(\mathrm 
{L}_i\mathrm 
{x}-\mathrm 
{r}_i) \numberthis\label{e:midual}
\end{cases}
\end{align*}
under the assumption that a solution exists.
\end{problem}

In the case when ${\mathrm X}={\mathrm H}$ and ${\mathrm 
C}_2=0$, 
Problem~\ref{prob:mi}  is studied in 
\cite{vu2013splitting}\footnote{Note that in \cite{vu2013splitting}, 
weights 
$(\omega_i)_{1\leq i\leq m}$ multiplying operators 
$(\mathrm{B}_i\infconv\mathrm{D}_i)_{1\leq i\leq m}$ are considered.
They can be retrieved in \eqref{e:miprimal} by considering 
$(\omega_i\mathrm{B}_i)_{1\leq i\leq m}$ and 
$(\omega_i\mathrm{D}_i)_{1\leq i\leq m}$ instead of 
$(\mathrm{B}_i)_{1\leq i\leq m}$ and 
$(\mathrm{D}_i)_{1\leq i\leq m}$. Then both formulations are 
equivalent.}
and
models a large class of problems including 
optimization problems, variational inequalities, equilibrium 
problems, among others (see 
{\color{red}\cite{briceno2011monotone+,he2012convergence,vu2013splitting,condat2013primal}}
 and 
the 
references therein). In \cite{vu2013splitting} the author 
rewrite \eqref{e:miprimal} and \eqref{e:midual} in the case 
${\mathrm X}={\mathrm H}$ as
\begin{equation}
\label{e:mipd}
\text{find}\quad {z}\in\HH\quad\text{such that 
}\quad 
{0}\in 
{M}{z}+{S}{z}+Q{z},
\end{equation}
where $\HH=\mathrm{H\times G}_1\times\cdots\times\mathrm{G}_m$, 
${M}\colon\HH\to 
2^{\HH}\colon (\mathrm{x},\mathrm{v}_1,\ldots, 
\mathrm{v}_m)\mapsto (\mathrm{Ax-z})\times 
(\mathrm{B}_1^{-1}\mathrm{v}_1+\mathrm{r}_1)\times\cdots\times 
(\mathrm{B}_m^{-1}\mathrm{v}_m+\mathrm{r}_m)$ is maximally monotone, 
${S}\colon 
\HH\to\HH\colon (\mathrm{x},\mathrm{v}_1,\ldots, 
\mathrm{v}_m)\mapsto 
(\sum_{i=1}^m\mathrm{L}_i^*\mathrm{v}_i,
-\mathrm{L}_1\mathrm{x},\ldots,-\mathrm{L}_m\mathrm{x})$ is skew linear,
 and
$Q\colon 
\HH\to\HH\colon 
(\mathrm{x},\mathrm{v}_1,\ldots,\mathrm{v}_m)\mapsto 
(\mathrm{C}_1\mathrm{x},\mathrm{D}_1^{-1}\mathrm{v}_1,\ldots,
\mathrm{D}_m^{-1}\mathrm{v}_m)$ is cocoercive.
 If $(\mathrm{x},\mathrm{v}_1,\ldots,\mathrm{v}_m)$ is a 
 solution in the primal-dual space $\HH$ to \eqref{e:mipd}, then
 $\mathrm{x}$ is a solution to  \eqref{e:miprimal} and 
 $(\mathrm{v}_1,\ldots,\mathrm{v}_m)$ is a solution to \eqref{e:midual}. The author 
 provide an algorithm for solving 
 \eqref{e:miprimal}--\eqref{e:midual} in this particular instance,
which is an application of the {\em 
 forward-backward} splitting (FBS) applied to the inclusion 
 \begin{equation}
 \label{e:mipdvm}
 \text{find}\quad {z}\in\HH\quad\text{such that 
 }\quad 
 {0}\in 
 V^{-1}({M}+{S}){z}+V^{-1}Q{z},
 \end{equation}
 where $V$ is a specific symmetric strongly monotone operator.
 Under the metric $\scal{V\cdot}{\cdot}$, $V^{-1}({M}+{S})$ is 
 maximally monotone and $V^{-1}Q$ is cocoercive and, therefore, 
 the FBS converges weakly to a primal-dual solution. 
 
 In order to tackle the case $\mathrm{C}_2\neq 0$, we propose to 
 use the method in Theorem~\ref{cor:asymmetricnoinversion} for solving 
 $0\in Ax+B_1x+B_2x$ where $A=M$, $B_1=Q$, $B_2=S+C_2$,
 and $C_2\colon(\mathrm{x},\mathrm{v}_1,\ldots,\mathrm{v}_m)\mapsto 
 (\mathrm{C}_2\mathrm{x},0,\ldots,
0)$
 allowing, in that way, non self-adjoint linear operators which may vary among 
 iterations. 
 The following result provides the method thus obtained, where the dependence 
 of the non self-adjoint linear operators with respect to iterations has been avoided for 
 simplicity.
 
 \begin{theorem}
 \label{p:primdumi} 
 In  Problem~\ref{prob:mi}, set $\mathrm{X}=\mathrm{H}$, set 
 $\mathrm{G}_0=\mathrm{H}$, for every $i\in\{0,1,\ldots,m\}$ and 
 $j\in\{0,\ldots, i\}$, let
 $\mathrm{P}_{ij}\colon\mathrm{G}_j\to\mathrm{G}_i$ be a linear 
 operator satisfying 
 \begin{equation}
 \label{e:condiPii}
 (\forall 
   \mathrm{x}_{i}\in\mathrm{G}_{i})\quad 
   \scal{\mathrm{P}_{ii}\mathrm{x}_{i}}{\mathrm{x}_{i}}
   \geq\varrho_i\|\mathrm{x}_{i}\|^2
   \end{equation}
   for some $\varrho_i>0$. Define the $(m+1)\times(m+1)$ 
   symmetric real 
   matrices $\Upsilon$, $\Sigma$, and $\Delta$ by
 \begin{align*}
  (\forall i\in\{0,\ldots,m\})(\forall j<i)\quad \Upsilon_{ij}&=
  \begin{cases}
   0,\quad&\text{if }i=j;\\
  \|\mathrm{P}_{ij}\|/2,&\text{if }i>j,\\
  \end{cases}
\quad \\
  \Sigma_{ij}&=
    \begin{cases}
 \|\mathrm{P}_{ii}-\mathrm{P}_{ii}^*\|/2,\quad&\text{if }i=j;\\
    \|\mathrm{L}_i+\mathrm{P}_{i0}/2\|,&\text{if }i\geq 1;j=0;\\
    \|\mathrm{P}_{ij}\|/2,&\text{if }i>j>0,
    \end{cases}\numberthis  \label{e:defUpsilon}
  \end{align*}
and $\Delta={\rm Diag}(\varrho_0,\ldots,\varrho_m)$.
Assume that $\Delta-\Upsilon$
is positive definite with smallest eigenvalue $\rho>0$ and
that 
\begin{equation}
\label{e:metricconditionpd}
(\|\Sigma\|_2+\delta)^2<\rho\left(\rho-\frac{1}{2\beta}\right),
\end{equation}
where $\beta=\min\{\mu,\nu_1,\ldots,\nu_m\}$.
Let 
$M=\max_{i=0,\ldots,m}\|\mathrm{P}_{ii}\|+\|\Upsilon\|_2$, let
$\lambda\in]0,M^{-1}[$,
 let $(\mathrm{x}^0,\mathrm{u}_1^0,\ldots,\mathrm{u}_m^0)\in 
 \mathrm{H\times 
 G}_1\times\cdots\times 
 \mathrm{G}_m$, and let $\{\mathrm{x}^k\}_{k\in\NN}$ 
 and
 $\{\mathrm{u}_i^k\}_{k\in\NN, 1\leq i\leq m}$ the sequences 
 generated by the 
 following 
 routine: 
 for every $k\in\NN$
 \begin{equation}
 \label{e:algomicomp1} 
\hspace{-.2cm}
 \begin{array}{l}
 \left\lfloor
 \begin{array}{l}
 \!\!\mathrm{y}^k 
 \!=\!J_{ 
  \mathrm{P}_{00}^{-1}\mathrm{A}}
  \left(\mathrm{x}^k-\mathrm{P}_{00}^{-1}\bigg(\mathrm{C}_1\mathrm{x}^k+
 \mathrm{C}_2\mathrm{x}^k+\sum_{i=1}^m
 \mathrm{L}_i^*\mathrm{u}_i^k\bigg)\right)\\
 \!\!\mathrm{v}_1^k \!=\! J_{\mathrm{P}_{11}^{-1}\mathrm{B}_1^{-1}}
\!\left(\mathrm{u}_1^k - 
   \mathrm{P}_{11}^{-1}\bigg(\mathrm{D}_1^{-1}\mathrm{u}_1^k-
    \mathrm{L}_1\mathrm{x}^k-\mathrm{P}_{10}(\mathrm{x}^k-\mathrm{y}^k)\bigg)\right)\\
  \!\!\mathrm{v}_2^k\!=\! J_{ \mathrm{P}_{22}^{-1}\mathrm{B}_2^{-1}}
\!\left(\mathrm{u}_2^k - 
    \mathrm{P}_{22}^{-1}\bigg(\mathrm{D}_2^{-1}\mathrm{u}_2^k-
     \mathrm{L}_2\mathrm{x}^k-\mathrm{P}_{20}(\mathrm{x}^k-\mathrm{y}^k)
     -\mathrm{P}_{21}(\mathrm{u}_1^k-\mathrm{v}_1^k)\bigg)\right)\\
     \vdots\\
  \!\!\mathrm{v}_m^k \!=\!J_{\mathrm{P}_{mm}^{-1}\mathrm{B}_m^{-1}}
\!\!\left(\!\mathrm{u}_m^k\! -\! 
 \mathrm{P}_{mm}^{-1}\bigg(\!\mathrm{D}_m^{-1}\mathrm{u}_m^k\!-\!
  \mathrm{L}_m\mathrm{x}^k\!-\!\mathrm{P}_{m0}(\mathrm{x}^k\!-\!\mathrm{y}^k)
  \!-\!\sum_{j=1}^{m-1}\!\mathrm{P}_{mj}(\mathrm{u}_j^k\!-\!\mathrm{v}_j^k)\bigg)\!\right)\\
 \!\!\mathrm{x}^{k+1}= 
 \mathrm{x}^k+\lambda\left(\mathrm{P}_{00}(\mathrm{y}^k-\mathrm{x}^k)+
 \left(\mathrm{C}_2\mathrm{x}^k-\mathrm{C}_2\mathrm{y}^k+\sum_{i=1}^m
  \mathrm{L}_i^*(\mathrm{u}_i^k-\mathrm{v}_i^k)\right)\right)\\
  \! \!\mathrm{u}_1^{k+1}=\mathrm{u}_1^k+\lambda\Big(\mathrm{P}_{10}
   (\mathrm{y}^k-\mathrm{x}^k)+
    \mathrm{P}_{11}(\mathrm{v}_1^k-\mathrm{u}_1^k)-
   \mathrm{L}_1(\mathrm{x}^k-\mathrm{y}^k)\Big)\\
     \vdots\\
   \!\! \mathrm{u}_m^{k+1}
   =\mathrm{u}_m^k+\lambda\left(\mathrm{P}_{m0}(\mathrm{y}^k-\mathrm{x}^k)+
        \sum_{j=1}^m\mathrm{P}_{mj}(\mathrm{v}_j^k-\mathrm{u}_j^k)-
       \mathrm{L}_m(\mathrm{x}^k-\mathrm{y}^k)\right). 
 \end{array}
 \right.\\[2mm]
 \end{array}
  \end{equation}
 Then there exists a primal-dual solution 
 $(\mathrm{x}^*,\mathrm{u}_1^*,\ldots,\mathrm{u}_m^*)\in 
  \mathrm{H\times 
  G}_1\times\cdots\times 
  \mathrm{G}_m$ to 
 Problem~\ref{prob:mi} such that $\mathrm{x}^k\rightharpoonup 
 \mathrm{x}^*$ and, for every $i\in\{1,\ldots,m\}$,
 $\mathrm{u}_i^k\rightharpoonup \mathrm{u}_i^*$.
 \end{theorem}
 
 \begin{proof} 
Consider the real Hilbert space $\HH=\mathrm{H\oplus 
G}_1\oplus\cdots\oplus 
\mathrm{G}_m$, where its scalar product and norm are denoted by 
$\pscal{\cdot}{\cdot}$ and $|||\cdot|||$, respectively, and 
$x=(\mathrm{x}_0,\mathrm{x}_1,\ldots,\mathrm{x}_m)$ and 
$y=(\mathrm{y}_0,\mathrm{y}_1,\ldots,\mathrm{y}_m)$ denote generic 
elements of $\HH$. Similarly as in 
\cite{vu2013splitting}, note that the set of 
 primal-dual solutions 
 $x^*=(\mathrm{x}^*,\mathrm{u}_1^*,\ldots,\mathrm{u}_m^*)\in\HH$ to 
 Problem~\ref{prob:mi} in the case $\mathrm{X}=\mathrm{H}$ coincides 
 with the set of 
 solutions to the monotone inclusion
 \begin{equation}
 \text{find}\quad x\in\HH\quad\text{such that}\quad 0\in Ax+B_1x+B_2x,
 \end{equation}
 where the operators ${A}\colon\HH\to 
 2^{\HH}$,  $B_1\colon \HH\to\HH$, and
 ${B}_2\colon \HH\to\HH$ {($\dom B_2=\HH$)} defined by
 \begin{equation}
 	\begin{cases}
 	A&\colon (\mathrm{x},\mathrm{v}_1,\ldots, 
 	\mathrm{v}_m)\mapsto (\mathrm{Ax-z})\times 
 	(\mathrm{B}_1^{-1}\mathrm{v}_1+\mathrm{r}_1)\times\cdots\times 
 	(\mathrm{B}_m^{-1}\mathrm{v}_m+\mathrm{r}_m)\\
 	B_1&\colon 
 	(\mathrm{x},\mathrm{v}_1,\ldots,\mathrm{v}_m)\mapsto 
 	(\mathrm{C}_1\mathrm{x},\mathrm{D}_1^{-1}\mathrm{v}_1,\ldots,
 	\mathrm{D}_m^{-1}\mathrm{v}_m)\\
 	B_2&\colon (\mathrm{x},\mathrm{v}_1,\ldots, 
 	\mathrm{v}_m)\mapsto 
 	(\mathrm{C}_2\mathrm{x}+\sum_{i=1}^m\mathrm{L}_i^*\mathrm{v}_i,
 	-\mathrm{L}_1\mathrm{x},\ldots,-\mathrm{L}_m\mathrm{x}),
 	\end{cases}
 \end{equation}
 are maximally monotone, 
 $\beta$-cocoercive, and monotone-Lipschitz, respectively 
{ (see  \cite[Proposition~20.22\,and\,20.23]{bauschke2017convex} and 
 \cite[Eq. (3.12)]{vu2013splitting}).}
 
 Now let
$P\colon\HH\to \HH$ defined by
\begin{equation}
\label{e:def_P}
P\colon x\mapsto \left(\mathrm{P}_{00}\mathrm{x}_0, 
\mathrm{P}_{10}\mathrm{x}_0+\mathrm{P}_{11}\mathrm{x}_1,\ldots,\sum_{j=0}^m
\mathrm{P}_{mj}\mathrm{x}_j\right)=\left(\sum_{j=0}^i
\mathrm{P}_{ij}\mathrm{x}_j\right)_{i=0}^m. 
\end{equation}
Then $P^*\colon x\mapsto(\sum_{j=i}^m
\mathrm{P}_{ji}^*\mathrm{x}_j)_{i=0}^m$ and 
$U\colon\HH\to \HH$ and $S\colon\HH\to \HH$ defined by
\begin{align}
\label{e:def_U}
U&\colon x\mapsto \left(\frac{1}{2}\sum_{j=0}^{i-1}
\mathrm{P}_{ij}\mathrm{x}_j+
\left(\frac{\mathrm{P}_{ii}+\mathrm{P}_{ii}^*}{2}\right)\mathrm{x}_i+
\frac{1}{2}\sum_{j=i+1}^m
\mathrm{P}_{ji}^*\mathrm{x}_j\right)_{i=0}^m\\
\label{e:def_S}
S&\colon x\mapsto \left(\frac{1}{2}\sum_{j=0}^{i-1}
\mathrm{P}_{ij}\mathrm{x}_j+
\left(\frac{\mathrm{P}_{ii}-\mathrm{P}_{ii}^*}{2}\right)\mathrm{x}_i-
\frac{1}{2}\sum_{j=i+1}^m
\mathrm{P}_{ji}^*\mathrm{x}_j\right)_{i=0}^m
\end{align}
are the self-adjoint and skew components of $P$, respectively, satisfying $P=U+S$.
 Moreover, for every
$x=(\mathrm{x}_0,\mathrm{x}_1,\ldots,\mathrm{x}_m)$ 
in $\HH$, we have
\begin{align}
\label{e:Ustronglymon}
\pscal{Ux}{x}&=\sum_{i=0}^m\frac{1}{2}\sum_{j=0}^{i-1}
\scal{\mathrm{P}_{ij}\mathrm{x}_j}{\mathrm{x}_i}+
\scal{\mathrm{P}_{ii}\mathrm{x}_i}{\mathrm{x}_i}+
\frac{1}{2}\sum_{j=i+1}^m
\scal{\mathrm{P}_{ji}^*\mathrm{x}_j}{\mathrm{x}_i}\nonumber\\
&=\sum_{i=0}^m\scal{\mathrm{P}_{ii}\mathrm{x}_i}{\mathrm{x}_i}
+\sum_{i=1}^m\sum_{j=0}^{i-1}\scal{\mathrm{P}_{ij}\mathrm{x}_j}{\mathrm{x}_i}
\nonumber\\
&\geq\sum_{i=0}^m\varrho_i\|\mathrm{x}_i\|^2-
\sum_{i=1}^m\sum_{j=0}^{i-1}\|\mathrm{P}_{ij}\|\,\|\mathrm{x}_i\|\,\|
\mathrm{x}_j\|
\nonumber\\
&=\xi\cdot(\Delta-\Upsilon)\xi\geq\rho|\xi|^2=\rho\,|||x|||^2,
\end{align}
where $\xi:=(\|\mathrm{x}_i\|)_{i=0}^m\in\RR^{m+1}$, $\Upsilon$ is 
defined in \eqref{e:defUpsilon}, and $\rho$ is the smallest (strictly 
positive) 
eigenvalue of $\Delta-\Upsilon$. In addition, we can write 
$B_2-S=C_2+R$,
where $C_2\colon x\mapsto 
  (\mathrm{C}_2\mathrm{x},0,\ldots,0)$ is monotone and 
  $\delta$-lipschitzian, and $R$ is a skew 
   linear operator satisfying, for every
$x=(\mathrm{x}_0,\mathrm{x}_1,\ldots,\mathrm{x}_m)\in\HH$,
$Rx=(\sum_{j=0}^mR_{i,j}\mathrm{x}_j)_{0\leq i\leq m}$, where the 
operators
$R_{i,j}\colon\mathrm{G}_j\to \mathrm{G}_i$ are defined by 
$R_{i,j}=-\mathrm{P}_{ij}/2$ if $i>j>0$, 
$R_{i,j}=-(\mathrm{L}_{i}+\mathrm{P}_{i0})/2$ if $i>j=0$, 
$R_{i,i}=(\mathrm{P}_{ii}^*-\mathrm{P}_{ii})/2$ and the other 
components follow from the skew property of $R$. 
Therefore,
\begin{align}
\label{e:desigS}
|||Rx|||^2\!=\!\sum_{i=0}^m\left\|\sum_{j=0}^mR_{i,j}\mathrm{x}_j\right\|^2
\!\!\!\leq\sum_{i=0}^m\left(\sum_{j=0}^m\|R_{i,j}\|\,\|\mathrm{x}_j\|\right)^{\!\!\!2}
\!\!=|\Sigma\xi|^2\!\leq\|\Sigma\|^2_2|\xi|^2\!=\|\Sigma\|^2_2|||x|||^2,
\end{align}
from which we obtain that
$B_2-S$ is $(\delta+\|\Sigma\|_2)$-lipschitzian. 
Altogether, by noting that, for every $x\in\HH$,
 $\|Ux\|\leq M$, all the 
hypotheses of Theorem~\ref{cor:asymmetricnoinversion} hold in this 
instance and by developing \eqref{eq:FBF-asymmetric-no-U} for this 
specific choices of $A$, $B_1$, $B_2$, $P$, $\gamma$, and setting, 
for every $k\in\NN$, 
$z^k=(\mathrm{x}^k,\mathrm{u}_1^k,\ldots,\mathrm{u}_m^k)$ and 
$x^k=(\mathrm{y}^k,\mathrm{v}_1^k,\ldots,\mathrm{v}_m^k)$, we obtain
\eqref{e:algomicomp1} after straighforward computations and  using
\begin{equation}
{ x^k=J_{ P^{-1}A}(z^k-P^{-1}(B_1z^k+B_2z^k))\quad
\Leftrightarrow\quad P(z^k-x^k)-(B_1z^k+B_2z^k)\in  Ax^k.}
\end{equation}
The result follows, hence, as a consequence of 
Theorem~\ref{cor:asymmetricnoinversion}.
\end{proof}

\begin{remark}
\begin{enumerate}
\item As in Theorem~\ref{cor:asymmetricnoinversion}, the algorithm in 
Theorem~\ref{p:primdumi} allows for linear operators
$(\mathrm{P}_{ij})_{0\leq i,j\leq m}$ depending on the iteration, 
whenever \eqref{e:metricconditions2} holds for the corresponding 
operators defined in 
\eqref{e:def_P}--\eqref{e:def_S}. We omit this generalization in 
Theorem~\ref{p:primdumi} for the sake of simplicity.

\item In the particular case when, for every $i\in\{1,\ldots,m\}$, 
$\textrm{B}_i=\widetilde{\textrm{B}}_i\infconv \textrm{M}_i$, where 
$\textrm{M}_i$ is such that $\textrm{M}_i^{-1}$ is monotone and 
$\sigma_i$-Lipschitz, for some $\sigma_i>0$, Problem~\eqref{prob:mi} 
can be solved 
in a similar way if, instead of $B_2$ and $\delta$, we consider 
$\widetilde{B}_2\colon 
(\mathrm{x},\mathrm{v}_1,\ldots, 
\mathrm{v}_m)\mapsto 
(\mathrm{C}_2\mathrm{x}+\sum_{i=1}^m\mathrm{L}_i^*\mathrm{v}_i,
\mathrm{M}_1^{-1}\mathrm{v}_1-\mathrm{L}_1\mathrm{x},\ldots,
\mathrm{M}_m^{-1}\mathrm{v}_m-\mathrm{L}_m\mathrm{x})$ and
$\widetilde{\delta}=\max\{\delta,\sigma_1,\ldots,\sigma_m\}$. Again, 
for the sake of 
simplicity, this extension has not been considered in 
Problem~\ref{prob:mi}.

\item If the inversion of the matrix $U$ is not difficult or no 
variable metric is used and the projection onto 
$\mathrm{X}\subset\mathrm{H}$ is computable, we can also use 
Theorem~\ref{thm:asymmetric_metric} for solving Problem~\ref{prob:mi} 
in the general case $\mathrm{X}\subset\mathrm{H}$.
 
\end{enumerate}
\end{remark}
 
\begin{corollary}
	\label{c:pd}
In Problem~\ref{prob:mi}, let $\theta\in[-1,1]$, let 
$\sigma_0,\ldots,\sigma_m$ be strictly positive real numbers and let 
$\Omega$ the 
$(m+1)\times(m+1)$ symmetric real matrix given by
\begin{equation}
\label{e:Omega}
(\forall i,j\in\{0,\ldots,m\})\quad \Omega_{ij}=
\begin{cases}
\frac{1}{\sigma_i},\quad&\text{if }i=j;\\
-(\frac{1+\theta}{2})\|\mathrm{L}_i\|,&\text{if }0=j<i;\\
0,&\text{if }0<j<i.
\end{cases}
\end{equation}
Assume that $\Omega$ is positive definite with $\rho>0$ its 
smallest eigenvalue and that
\begin{equation}
\label{e:conditioncor}
\left(\delta+\left(\frac{1-\theta}{2}\right)
\sqrt{\sum_{i=1}^m\|\mathrm{L}_i\|^2}\right)^2
<\rho\left(\rho-\frac{1}{2\beta}\right),
\end{equation}
where $\beta=\min\{\mu,\nu_1,\ldots,\nu_m\}$.
Let 
$M=(\min\{\sigma_0,\ldots,\sigma_m\})^{-1}+
(\frac{1+\theta}{2})\sqrt{\sum_{i=1}^m\|\mathrm{L}_{i}\|^2}$,
 let
$\lambda\in]0,M^{-1}[$,
 let $(\mathrm{x}^0,\mathrm{u}_1^0,\ldots,\mathrm{u}_m^0)\in 
 \mathrm{H\times 
 G}_1\times\cdots\times 
 \mathrm{G}_m$, and let $\{\mathrm{x}^k\}_{k\in\NN}$ 
 and
 $\{\mathrm{u}_i^k\}_{k\in\NN, 1\leq i\leq m}$ the sequences 
 generated by the 
 following 
 routine: 
\begin{equation}
\label{e:genVU}
(\forall k\in\NN)\quad 
\begin{array}{l}
 \left\lfloor
 \begin{array}{l}
\mathrm{y}^k 
 =J_{\sigma_0\mathrm{A}}
  \left(\mathrm{x}^k-\sigma_0\left(\mathrm{C}_1\mathrm{x}^k+
 \mathrm{C}_2\mathrm{x}^k+\sum_{i=1}^m
 \mathrm{L}_i^*\mathrm{u}_i^k\right)\right)\\[2mm]
 \text{For every } i=1,\ldots,m\\
\left\lfloor
 \mathrm{v}_i^k = J_{ 
\sigma_i\mathrm{B}_i^{-1}}\left(\mathrm{u}_i^k - 
\sigma_i\left(\mathrm{D}_i^{-1}\mathrm{u}_i^k-
    \mathrm{L}_i(\mathrm{y}^k+\theta(\mathrm{y}^k
    -\mathrm{x}^k))\right)\right)
    \right.\\[2mm]
  \mathrm{x}^{k+1}= 
 \mathrm{x}^k+\frac{\lambda}{\sigma_0}\left(\mathrm{y}^k-\mathrm{x}^k+
 \sigma_0\left(\mathrm{C}_2\mathrm{x}^k-\mathrm{C}_2\mathrm{y}^k+\sum_{i=1}^m
  \mathrm{L}_i^*(\mathrm{u}_i^k-\mathrm{v}_i^k)\right)\right)\\
 \text{For every } i=1,\ldots,m\\
\left\lfloor
 \mathrm{u}_i^{k+1}=\mathrm{u}_i^k+\frac{\lambda}{\sigma_i}
\left(\mathrm{v}_i^k-\mathrm{u}_i^k-
\sigma_i\theta\mathrm{L}_i(\mathrm{y}^k-\mathrm{x}^k)\right),
    \right.\\[2mm]
\end{array}
 \right.\\[2mm]
 \end{array}
  \end{equation}
Then there exists a primal-dual solution 
 $(\mathrm{x}^*,\mathrm{u}_1^*,\ldots,\mathrm{u}_m^*)\in 
  \mathrm{H\times 
  G}_1\times\cdots\times 
  \mathrm{G}_m$ to 
 Problem~\ref{prob:mi} such that $\mathrm{x}^k\rightharpoonup 
 \mathrm{x}^*$ and, for every $i\in\{1,\ldots,m\}$,
 $\mathrm{u}_i^k\rightharpoonup \mathrm{u}_i^*$.
\end{corollary}
\begin{proof}
This result is a consequence of Theorem~\ref{p:primdumi} when, for 
every 
$i\in\{0,\ldots,m\}$, $\mathrm{P}_{ii}=\Id/\sigma_i$, 
$\mathrm{P}_{i0}=-(1+\theta)\mathrm{L}_{i}$,
and, for every $0<j<i$, $\mathrm{P}_{ij}=0$. Indeed, 
we have from \eqref{e:condiPii} that
$\varrho_i=1/\sigma_i$, and from \eqref{e:defUpsilon} we deduce that,
for every $x=(\xi_i)_{0\leq i\leq m}\in\RR^{m+1}$,
\begin{equation}
\|\Sigma 
x\|^2=\left(\frac{1-\theta}{2}\right)^2\left[\left(\sum_{i=0}^m\|\mathrm{L}_i\|\xi_i\right)^2+\xi_0^2
\sum_{i=1}^m\|\mathrm{L}_i\|^2\right]\leq 
\left(\frac{1-\theta}{2}\right)^2\left(\sum_{i=0}^m\|\mathrm{L}_i\|^2\right)\|x\|^2,
\end{equation}
from which we obtain 
$\|\Sigma\|_2\leq 
(\frac{1-\theta}{2})\sqrt{\sum_{i=1}^m\|\mathrm{L}_i\|^2}$. Actually, 
we have the equality by choosing $\bar{x}=(\bar{\xi}_i)_{0\leq i\leq 
m}$ defined 
by 
$\bar{\xi}_i=\|\mathrm{L}_i\|/\sqrt{\sum_{j=1}^m\|\mathrm{L}_j\|^2}$ 
for 
every $i\in\{1,\ldots,m\}$ and $\bar{\xi}_0=0$, which satisfies 
$\|\bar{x}\|=1$ 
and $\|\Sigma 
\bar{x}\|=(\frac{1-\theta}{2})\sqrt{\sum_{i=1}^m\|\mathrm{L}_i\|^2}$.
Therefore, condition 
\eqref{e:metricconditionpd} reduces to \eqref{e:conditioncor}.
On the other hand, from \eqref{e:defUpsilon} we deduce that 
$\Omega=\Delta-\Upsilon$ and 
$\Upsilon=(\frac{1+\theta}{1-\theta})\Sigma$, which yields 
$\|\Upsilon\|_2=(\frac{1+\theta}{2})\sqrt{\sum_{i=1}^m\|\mathrm{L}_i\|^2}$
and $\max_{i=0,\ldots,m}\|\mathrm{P}_{ii}\|
=(\min\{\sigma_0,\ldots,\sigma_m\})^{-1}$. Altogether, since 
\eqref{e:genVU} is exactly \eqref{e:algomicomp1} for this choice of 
matrices $(\mathrm{P}_{i,j})_{0\leq i,j,\leq m}$, the result is a 
consequence of Theorem~\ref{p:primdumi}.
\end{proof}

\begin{remark}
\label{r:pd2}
\begin{enumerate}
\item\label{r:pd21} Note that, the condition $\rho>0$ where $\rho$ is the smallest 
eigenvalue of $\Omega$ defined in \eqref{e:Omega}, is 
guaranteed if  
$\sigma_0(\frac{1+\theta}{2})^2\sum_{i=1}^m\sigma_i\|\mathrm{L}_i\|^2<1$.
Indeed, by repeating the procedure in \cite[(3.20)]{vu2013splitting} 
in finite dimension we obtain, for every $x=(\xi_i)_{0\leq i\leq 
m}\in\RR^{m+1}$,
\begin{align}
\hspace{-.5cm}x\cdot\Omega x&=
\sum_{i=0}^m\frac{\xi_i^2}{\sigma_i}-\sum_{i=1}^m
2\left(\frac{1+\theta}{2}\right)\xi_0\|\mathrm{L}_i\|\xi_i\nonumber\\
&=\sum_{i=0}^m\frac{\xi_i^2}{\sigma_i}-\left(\frac{1+\theta}{2}\right)
\sum_{i=1}^m2
\frac{\sqrt{\sigma_i\|\mathrm{L}_i\|\xi_0}}{(\sigma_0\sum_{j=1}^m
\sigma_j\|\mathrm{L}_j\|^2)^{1/4}}
\frac{(\sigma_0\sum_{j=1}^m
\sigma_j\|\mathrm{L}_j\|^2)^{1/4}\xi_i}{\sqrt{\sigma_i}}\label{e:necess}\\
&\ge\sum_{i=0}^m\frac{\xi_i^2}{\sigma_i}-\left(\frac{1+\theta}{2}\right)
\left(\frac{\xi_0^2}{\sqrt{\sigma_0}}\sqrt{\sum_{j=1}^m\sigma_j\|\mathrm{L}_j\|^2}
+\sqrt{\sigma_0\sum_{j=1}^m\sigma_j\|\mathrm{L}_j\|^2}
\sum_{j=1}^m\frac{\xi_j^2}{\sigma_j}\right)\nonumber\\
&=\left(1-\left(\frac{1+\theta}{2}\right)\sqrt{\sigma_0\sum_{j=1}^m\sigma_j\|\mathrm{L}_j\|^2}\right)
\sum_{i=0}^m\frac{\xi_i^2}{\sigma_i}\nonumber\\
&\ge \rho_v\|x\|^2\label{e:ineqrho}
\end{align}
with 
\begin{equation}
\label{e:rhov}
\rho_v=\max\{\sigma_0,\ldots,\sigma_m\}^{-1}\left(1-
\left(\frac{1+\theta}{2}\right)
\sqrt{\sigma_0\sum_{j=1}^m\sigma_j\|\mathrm{L}_j\|^2}\right).
\end{equation}
Note that $\rho_v$ coincides with the constant obtained in  
\cite{vu2013splitting} in the case $\theta=1$ and we have $\rho\geq 
\rho_v$. Moreover, 
$\sigma_0(\frac{1+\theta}{2})^2\sum_{i=1}^m\sigma_i\|\mathrm{L}_i\|^2<1$
is also necessary for obtaining $\rho>0$, since in \eqref{e:necess} 
we can choose a particular vector $x$ for 
obtaining the equality. Of course, this choice
does not guarantee to also have equality in the last inequality in 
\eqref{e:ineqrho} and, hence, $\rho\geq\rho_v$ in general.

\item\label{r:pd22} If we set $\theta=1$ and
$\mathrm{C}_2=0$ and, hence, $\delta=0$,
\eqref{e:conditioncor} reduces to $2\beta\rho>1$
and we obtain from \eqref{e:genVU} a variant of 
\cite[Theorem~3.1]{vu2013splitting} 
including an extra forward step involving only the operators
$(\mathrm{L}_i)_{1\leq i\leq m}$. 
However, our condition is less restrictive, since $\rho\geq\rho_v$, 
where $\rho_v$ is defined in \eqref{e:rhov} and it is obtained in  
\cite{vu2013splitting} as we have seen in the last remark. Actually, 
in the particular case when $m=1$, 
$\mathrm{L}_1=\alpha\Id$, $\sigma_0=\eta^2\sigma_1=:\eta\sigma$ for 
some $0<\eta<1$, constants $\rho_v$ and $\rho$ reduce to
$$\rho_v(\eta)=\frac{1-\eta\sigma\alpha}{\sigma}\quad\text{and}\quad
\rho(\eta)=\frac{1}{2\sigma}\left(\frac{\eta^2+1}{\eta^2}-\sqrt{\left(\frac{\eta^2-1}
{\eta^2}\right)^2+4\alpha^2\sigma^2}\right),$$
respectively. By straightforward computations we deduce that 
$\rho(\eta)>\rho_v(\eta)$ for every
$0<\eta<(\alpha\sigma)^{-1}$, and hence our constant can strictly 
improve the condition $2\beta\rho>1$, needed in both approaches. 
Moreover, since Theorem~\ref{p:primdumi} allows for non self-adjoint 
linear operators varying among iterations, we can permit variable 
stepsizes $\sigma_{0}^k,\ldots,\sigma_{m}^k$ in 
Theorem~\ref{p:primdumi}, which could not 
be used in \cite{vu2013splitting} because of the variable metric 
framework.
\item\label{r:pd23} In the particular case when $\mathrm{C}_1=0$ and 
$\mathrm{C}_2=0$ we can take $\beta\to+\infty$ and, hence, condition 
\eqref{e:conditioncor} reduces to 
\begin{equation}
\label{e:conditioncorem}
\left(\frac{1-\theta}{2}\right)
\sqrt{\sum_{i=1}^m\|\mathrm{L}_i\|^2}
<\rho,
\end{equation}
which is stronger than the condition in \cite{he2012convergence}
for the case $m=1$, in which it is only needed that $\rho>0$ for 
achieving convergence. Indeed, in the case $m=1$, 
\eqref{e:conditioncorem} reduces to 
$2-2\theta\sigma_0\sigma_1\|\mathrm{L}_1\|^2>
(1-\theta)(\sigma_0+\sigma_1)\|\mathrm{L}_1\|$,
which coincides with the condition in \cite{he2012convergence} in the 
case $\theta=1$, but they differ if $\theta\neq 1$ 
because of the extra forward step coming from the Tseng's splitting 
framework. Actually, in the case $\theta=0$ it reduces to 
$\sigma_0+\sigma_1<2/\|\mathrm{L}_1\|$ and in the case $\theta=-1$
we obtain the stronger condition 
$\max\{\sigma_0,\sigma_1\}<1/\|\mathrm{L}_1\|$.
Anyway, in our context we can use constants 
$\sigma_0^k,\ldots,\sigma_m^k$ varying among iterations and we have a 
variant of the method in \cite{he2012convergence} and, in the case 
when $\theta=1$, of Chambolle-Pock's splitting 
\cite{chambolle2011first}.

\item\label{r:pd24} Since $\rho_v$ defined in \eqref{e:rhov} satisfies $\rho_v\leq 
\rho$ in the case when $\mathrm{C}_1=\mathrm{C}_2=0$, a sufficient condition for 
guaranteeing \eqref{e:conditioncorem} is 
$(1-\theta)\sqrt{\sum_{i=1}^m\|\mathrm{L}_i\|^2}/2<\rho_v$, which implied 
by the condition
\begin{equation}
	 \max\{\sigma_0,\ldots,\sigma_m\}\sqrt{\sum_{i=1}^m\|\mathrm{L}_i\|^2}<1.
\end{equation}

\item\label{r:pd25} Consider the case of composite optimization problems, i.e.,
when $\mathrm{A}=\partial \mathrm{f}$, $\mathrm{C}_1=\nabla \mathrm{h}$ 
for every $i=1,\ldots,m$, $\mathrm{B}_i=\partial 
\mathrm{g}_i$ and $\mathrm{D}_i=\partial\mathrm{\ell}_i$, where, for 
every 
$i=1,\ldots,m$,
$\mathrm{f}\colon\mathrm{H}\to\RX$ and 
$\mathrm{g}_i\colon\mathrm{G}_i\to\RX$ are proper lower 
semicontinuous and convex functions and 
$\mathrm{h}\colon\mathrm{H}\to\RR$ is differentiable, convex, with 
$\beta^{-1}$-Lipschitz gradient. In this case, any solution to 
Problem~\ref{prob:mi} when $\mathrm{C}_2=0$ is a solution to the 
primal-dual optimization problems
\begin{equation}
\min_{\mathrm{x}\in\mathrm{H}}{\mathrm{f}(\mathrm{x})+\mathrm{h}(\mathrm{x})
+\sum_{i=1}^m(\mathrm{g}_i\infconv\mathrm{\ell}_i)(\mathrm{L}_i\mathrm{x})}
\end{equation}
and
\begin{equation}
\min_{\mathrm{u}_1\in\mathrm{G}_1,\ldots,\mathrm{u}_m\in\mathrm{G}_m}
{(\mathrm{f}^*\infconv\mathrm{h}^*)
\left(-\sum_{i=1}^m\mathrm{L}_i^*\mathrm{u}_i\right)+
\sum_{i=1}^m\mathrm{g}_i^*(\mathrm{u}_i)+\mathrm{\ell}_i^*(\mathrm{u}_i)},
\end{equation}
and the equivalence holds under some qualification condition. In this particular case,
\eqref{e:genVU} reduces to
\begin{equation}
	\label{e:genVUopti}
	\begin{array}{l}
		\left\lfloor
		\begin{array}{l}
			\mathrm{y}^k 
			=\prox_{\sigma_0\mathrm{f}}
			\left(\mathrm{x}^k-\sigma_0\left(\nabla\mathrm{h}(\mathrm{x}^k)+
			\sum_{i=1}^m
			\mathrm{L}_i^*\mathrm{u}_i^k\right)\right)\\[2mm]
			\text{For every } i=1,\ldots,m\\
			\left\lfloor
			\mathrm{v}_i^k = \prox_{ 
				\sigma_i\mathrm{g}_i^*}\left(\mathrm{u}_i^k - 
			\sigma_i\left(\nabla\mathrm{\ell}_i^*(\mathrm{u}_i^k)-
			\mathrm{L}_i(\mathrm{y}^k+\theta(\mathrm{y}^k
			-\mathrm{x}^k))\right)\right)
			\right.\\[2mm]
			\mathrm{x}^{k+1}= 
			\mathrm{x}^k+\frac{\lambda}{\sigma_0}\left(\mathrm{y}^k-\mathrm{x}^k+
			\sigma_0\sum_{i=1}^m
			\mathrm{L}_i^*(\mathrm{u}_i^k-\mathrm{v}_i^k)\right)\\
			\text{For every } i=1,\ldots,m\\
			\left\lfloor
			\mathrm{u}_i^{k+1}=\mathrm{u}_i^k+\frac{\lambda}{\sigma_i}
			\left(\mathrm{v}_i^k-\mathrm{u}_i^k-
			\sigma_i\theta\mathrm{L}_i(\mathrm{y}^k-\mathrm{x}^k)\right),
			\right.\\[2mm]
		\end{array}
		\right.\\[2mm]
	\end{array}
\end{equation}
which, in the case $m=1$, is very similar to the method proposed in 
\cite[Algorithm~3]{patrinos2016asym} (by taking $\mu=(1-\theta)^{-1}$ for 
$\theta\in[-1,0]$), with a slightly different choice of the 
parameters involved in the last two lines in \eqref{e:genVUopti}. {On the other 
hand, in the case when $\ell=0$ and $\theta=1$, it differs from 
\cite[Algorithm~5.1]{condat2013primal} in the last two steps, in which linear operators are 
involved in our case. } An advantage of our 
method, even in the case $m=1$, is that the stepsizes $\sigma_0$ and $\sigma_1$ 
may vary among iterations.
\end{enumerate}

\end{remark}

\section{Applications}
\label{sec:6}
In this section we explore {four} applications for illustrating the advantages and 
flexibility 
of the methods proposed in the previous sections. In the first application, we apply  
Theorem~\ref{t:1} to the obstacle problem in PDE's in which dropping the extra forward 
step decreases the computational cost per iteration because
the computation of an extra gradient step is numerically expensive.
In the second application, devoted to empirical risk minimization (ERM), we illustrate 
the flexibility of using non self-adjoint linear operators. We derive different sequential 
algorithms depending on the nature of the linear operator involved. { In the third 
application, we develop a distributed operator-splitting scheme which allows for 
time-varying communication graphs. Finally, the last application focuses in nonlinear 
constrained optimization, in which monotone non-Lipschitz operators arise naturally. }

\subsection{Obstacle problem}

The obstacle problem is to find the equilibrium position of an elastic membrane on a 
domain $\Omega$, whose boundary is fixed and is restricted to remain above the 
some obstacle, given by the function $\varphi\colon\Omega\to\RR$. This problem can 
be applied to fluid filtration in porous media, elasto-plasticity, optimal control among 
other disciplines (see, e.g., \cite{Caf88} and the references therein).
Let $u\colon\Omega\to\RR$ be a function representing the vertical displacement of the 
membrane and let $\psi\colon\Gamma\to\RR$ be the function representing the fixed 
boundary, where $\Gamma$ is the smooth boundary of $\Omega$. Assume that
$\psi\in H^{1/2}(\Gamma)$ and $\varphi\in C^{1,1}(\Omega)$ satisfy
$\mathrm{T}\varphi\leq\psi$, and consider the 
problem 
	\begin{align}
		\min_{\mathrm{u}\in H^1(\Omega)}&\frac{1}{2}\int_{\Omega}|\nabla 
		\mathrm{u}|^2\mathrm{dx}\nonumber\label{e:obstprob}\\
		\text{s.t. } \mathrm{T}\mathrm{u}&=\psi,\quad\text{a.e. on 
		}\Gamma;	\\
		\mathrm{u}&\geq \varphi,\quad\text{a.e. in }\Omega\nonumber,
	\end{align}
	where $\mathrm{T}\colon H^1(\Omega)\to H^{1/2}(\Gamma)$ is 
	the (linear) trace operator and $H^1(\Omega)$ is endowed with the 
	scalar product
	$\scal{\cdot}{\cdot}\colon(\mathrm{u},\mathrm{v})\mapsto
	\int_{\Omega}\mathrm{u}\mathrm{v}\,dx+\int_{\Omega}\nabla\mathrm{u}
	\cdot\nabla\mathrm{v}\,dx$. There is a unique solution to this 
	obstacle problem \cite{Caff98}.
	
	In order to set this problem in our context, let us define the 
	operator
	\begin{equation}
	\label{e:Q}
	Q\colon H^{-1}(\Omega)\times H^{-1/2}(\Gamma)\to H^1(\Omega)
	\end{equation}
	which associates to each $(\mathrm{q},\mathrm{w})\in 
	H^{-1}(\Omega)\times H^{-1/2}(\Gamma)$ the unique weak solution 
	(in the sense of 
	distributions) to \cite[Section~25]{ZeidlerIIB}
	\begin{equation}
	\begin{cases}
	-\Delta \mathrm{u}+\mathrm{u}=\mathrm{q},\quad&\text{in 
	}\Omega;\\
	\frac{\partial \mathrm{u}}{\partial \nu}=\mathrm{w},&\text{on 
	}\Gamma,
	\end{cases}
	\end{equation}
	where $\nu$ is outer unit vector normal to $\Gamma$. Hence, $Q$ 
	satisfies
	\begin{equation}
	\label{e:Qweak}
	(\forall \mathrm{v}\in 
	\mathrm{H})\qquad\scal{Q(\mathrm{q},\mathrm{w})}{\mathrm{v}}=
	\scal{\mathrm{w}}{\mathrm{T}\mathrm{v}}_{{-1/2},{1/2}}
	+\scal{\mathrm{q}}{\mathrm{v}}_{{-1},1},
	\end{equation}
	where $\scal{\cdot}{\cdot}_{{-1/2},{1/2}}$ and 
	$\scal{\cdot}{\cdot}_{{-1},{1}}$ stand for the dual pairs   
	$H^{-1/2}(\Gamma)-H^{1/2}(\Gamma)$ and 
	$H^{-1}(\Omega)-H^{1}(\Omega)$, respectively. Then, by defining
	$\mathrm{H}=H^1(\Omega)$, $\mathrm{G}=H^{1/2}(\Gamma)$, 
	$\mathrm{f}\colon \mathrm{u}\mapsto 
	\frac{1}{2}\int_{\Omega}|\nabla 
	\mathrm{u}|^2\mathrm{dx}$, $\mathrm{g}=\iota_{\mathrm{C}}$, 
	where 
	$\mathrm{C}=\menge{\mathrm{u}\in\mathrm{H}}{\mathrm{u}\geq\varphi\;\text{
			a.e. 
			in }\Omega}$, let $\mathrm{D}={\{\psi\}}$, and 
	let $\mathrm{L}=\mathrm{T}$, \eqref{e:obstprob} can be 
	written equivalently as
	\begin{equation}
	\label{e:probobstheo}
	\min_{\mathrm{L}\mathrm{u}\in\mathrm{D}}\mathrm{f}(\mathrm{u})+
	\mathrm{g}(\mathrm{u}).
	\end{equation}
	Moreover, it is easy to verify that $\mathrm{f}$ is convex and, by 
	using integration by parts and 
	\eqref{e:Qweak}, for every 
	$\mathrm{h}\in\mathrm{H}$ we have
	\begin{align}
		\mathrm{f}(\mathrm{u}+\mathrm{h})-\mathrm{f}(\mathrm{u})-
		\Scal{\!Q\left(\!-\Delta\mathrm{u},\frac{\partial\mathrm{u}}{\partial\nu}\right)\!}{\!\mathrm{h}}
		&=\frac{1}{2}\int_{\Omega}|\nabla\mathrm{h}|^2dx+\int_{\Omega}\!\!
		\nabla\mathrm{u}\cdot\nabla\mathrm{h}\,dx+\scal{\Delta\mathrm{u}}{\mathrm{h}}_{-1,1}\nonumber\\
		&\hspace{20pt}-\scal{\frac{\partial\mathrm{u}}{\partial\nu}}{\mathrm{T}\mathrm{h}}_{-1/2,1/2}\nonumber\\
		&=\frac{1}{2}\int_{\Omega}|\nabla\mathrm{h}|^2dx,
	\end{align}
	which yields
	\begin{equation}
	\lim_{\|\mathrm{h}\|\to0}\frac{\left|\mathrm{f}(\mathrm{u}+\mathrm{h})-\mathrm{f}(\mathrm{u})-
		\Scal{Q\left(-\Delta\mathrm{u},\frac{\partial\mathrm{u}}{\partial\nu}\right)}
		{\mathrm{h}}\right|}{\|\mathrm{h}\|}=\frac{1}{2}\lim_{\|\mathrm{h}\|\to0}\frac{\|\nabla
		\mathrm{h}\|_{L^2}^2}{\|\mathrm{h}\|}=0.
	\end{equation}
	Hence, $\mathrm{f}$ is Fr\'echet differentiable with a linear gradient given by
	$\nabla\mathrm{f}\colon 
	\mathrm{u}\mapsto Q\left(\!-\Delta\mathrm{u},\frac{\partial\mathrm{u}}
	{\partial\nu}\right)$. Moreover, from integration by parts we have
	\begin{equation}
	\Scal{Q\left(-\Delta\mathrm{u},\frac{\partial\mathrm{u}}{\partial\nu}\right)}{\mathrm{h}}=
	\Scal{\frac{\partial\mathrm{u}}{\partial\nu}}{\mathrm{T}\mathrm{h}}_{-1/2,1/2}
	-\scal{\Delta\mathrm{u}}{\mathrm{h}}_{-1,1}=\int_{\Omega}
	\nabla\mathrm{u}\cdot\nabla\mathrm{h}\,dx\leq\|\mathrm{u}\|\|\mathrm{h}\|,
	\end{equation}
	which yields $\|\nabla\mathrm{f}(\mathrm{u})\|\leq\|\mathrm{u}\|$ 
	and, hence, it is $1$-cocoercive \cite{baillon1977quelques}. In addition, the trace 
	operator is 
	linear and bounded \cite{Grisv86} and we have from \eqref{e:Qweak} 
	that
	\begin{equation}
	\label{e:Qweak2}
	(\forall \mathrm{v}\in 
	\mathrm{H})(\forall \mathrm{w}\in 
	H^{1/2}(\Gamma))\qquad\scal{Q(0,\mathrm{w})}{\mathrm{v}}=
	\scal{\mathrm{w}}{\mathrm{T}\mathrm{v}}_{{-1/2},{1/2}},
	\end{equation}
	which yields 
	$\mathrm{L}^*\colon\mathrm{w}\mapsto\mathrm{Q}(0,\mathrm{w})$ and 
	since $\mathrm{C}$ is non-empty closed convex, $\mathrm{g}$ is 
	convex, proper, lower semicontinuous and $\prox_{\gamma 
		\mathrm{g}}=P_{\mathrm{C}}$, for any $\gamma >0$. 

Since first order conditions of \eqref{e:probobstheo} reduce to find 
	$(\mathrm{u},\mathrm{w})\in\mathrm{H}\times\mathrm{G}$ such that
	$0\in N_{\mathrm{C}}(\mathrm{u})+\nabla\mathrm{f}(\mathrm{u})+
	\mathrm{T}^*N_{\mathrm{D}}(\mathrm{T}\mathrm{u})$, which is a particular case 
	of Problem~\ref{prob:mi} and from Corollary~\ref{c:pd} when $\theta=1$ the method
\begin{equation}
\label{e:genVUobst}
\begin{array}{l}
\left\lfloor
\begin{array}{l}
\mathrm{v}^k 
=P_{\mathrm{C}}
\left(\mathrm{u}^k-\sigma_0Q\left(-\Delta\mathrm{u}^k,
\frac{\partial\mathrm{u}^k}{\partial\nu}+\mathrm{w}^k\right)\right)\\[2mm]
\mathrm{t}^k = \mathrm{w}^k +
\sigma_1
\left(\mathrm{T}(2\mathrm{y}^k-\mathrm{x}^k)-\psi\right)\\[2mm]
\mathrm{u}^{k+1}= 
\mathrm{u}^k+\frac{\lambda}{\sigma_0}\left(\mathrm{v}^k-\mathrm{u}^k+
\sigma_0
Q(0,\mathrm{w}^k-\mathrm{t}^k)\right)\\[2mm]
\mathrm{w}^{k+1}=\mathrm{w}^k+\frac{\lambda}{\sigma_1}
\left(\mathrm{t}^k-\mathrm{w}^k-
\sigma_1\mathrm{T}(\mathrm{v}^k-\mathrm{u}^k)\right)\\[2mm]
\end{array}
\right.\\[2mm]
\end{array}
\end{equation}
generates a weakly convergent sequence $(\mathrm{u}^k)_{k\in\NN}$ to the 
unique 
solution to the obstacle problem provided, for instance  (see 
Remark~\ref{r:pd2}.\ref{r:pd21}), that
$\max\{\sigma_0,\sigma_1\}+2\sqrt{\sigma_0\sigma_1}\|\mathrm{T}\|<2$.
Note that $\nabla\mathrm{f}$ must be computed only once at each iteration, improving 
the performance with respect to primal-dual methods following Tseng's approach, in 
which $\nabla\mathrm{f}$ must be computed twice by iteration (see, e.g., 
\cite{briceno2011monotone+,tseng2000modified}). The method proposed in 
{\cite{condat2013primal,vu2013splitting}} can also solve this problem but with 
stronger 
conditions on 
constants $\sigma_0$ and $\sigma_1$ as studied in Remark~\ref{r:pd2}. Moreover,
our approach may include variable stepsizes together with different assymetric linear 
operators which may improve the performance of the method. 

On the other hand, the general version of our method in 
Theorem~\ref{thm:asymmetric_metric} allows for an 
additional projection onto a closed convex set. In this case this can be useful to impose
some of the constraints of the problem in order to guarantee that iterates at each 
iteration satisfy such constraints. An additional projection step may accelerate the 
method as it has been studied in \cite{BAKS16}.
Numerical comparisons 
among these methods are part of further research.
	
\subsection{An Incremental Algorithm for Nonsmooth Empirical Risk 
Minimization}

In machine learning~\cite{Shalev-Shwartz:2014:UML:2621980}, the Empirical Risk 
Minimization (ERM) problem seeks to minimize a finite sample approximation of an 
expected loss, under conditions on the feasible set and the loss function. If the solution
to the sample approximation converges to a minimizer of the expected loss when the 
size of the sample increases, we say that the problem is learnable.
Suppose that we have a sample of size $m$, and, for every $i\in\{1,\ldots,m\}$,
 the loss function associated to the sample $\mathrm{z}_i$ is given by $l(\cdot; 
 \mathrm{z}_i) \colon 
 \mathrm{x}\mapsto  \mathrm{f}_i(\mathrm{a}_i^{\top} x)$, where each 
 $\mathrm{a}_i \in \RR^{d}\backslash \{0\}$ 
 and each $\mathrm{f}_i : \RR \rightarrow (-\infty, \infty]$ is closed, proper, and 
 convex. 
 Then the ERM problem is to

\begin{align}\label{eq:ERM}
\Min_{\mathrm{x} \in \RR^d} \frac{1}{m} \sum_{i=1}^{m} 
\mathrm{f}_i(\mathrm{a}_i^{\top} \mathrm{x}).
\end{align}
 This form features in support vector machines, logistic regression, linear regression, least-absolute deviations, and many other common models in machine learning. 

The parameter $m$ indicates the size of the training set and is typically large. 
Parallelizing a (sub)gradient computation of~\eqref{eq:ERM} is straightforward, but in 
general, because training sets are large, we may not have enough processors to do so. 
Thus, when only a few processors are available, incremental iterative algorithms, in 
which one or a few training samples are used per iteration to update our solution 
estimate, are a natural choice.

Several incremental algorithms are available for solving~\eqref{eq:ERM}, including 
incremental (sub)gradient descent and incremental aggregated gradient 
methods~\cite{schmidt2013minimizing,defazio2014finito,johnson2013accelerating,defazio2014saga,bertsekas2015incremental,wang2013incremental,bertsekasincrementalproximal,doi:10.1137/S1052623499362111,bianchi2015ergodic}.
 The former class requires 
diminishing stepsizes (e.g., of size 
$O(k^{-1/2})$) and, hence, their convergence may be very slow, while the latter 
class of algorithms is 
usually restricted to the cases in which either $\mathrm{f}_i$ is smooth or the dual 
problem 
of~\eqref{eq:ERM} is smooth (in which case~\eqref{eq:ERM} is strongly convex). In 
contrast, we now develop an 
incremental proximal algorithm, which imposes no smoothness or 
strong convexity assumptions. It has a Gauss-Seidel 
structure and is obtained by an application of 
Theorem~\ref{p:primdumi}. The involved 
stepsizes may vary among iterations but they are set to be constants 
for simplicity.

The method follows from the following first-order optimality 
conditions obtained 
assuming some qualification condition:
\begin{equation}
	\mathrm{x}\quad\text{solves } \eqref{eq:ERM}\quad\Leftrightarrow\quad 
	0\in\sum_{i=1}^{m}\mathrm{a}_i\partial 
	\mathrm{f}_i(\mathrm{a}_i^{\top}\mathrm{x}),
\end{equation}
which is a particular case of Problem~\ref{prob:mi} when 
$\mathrm{H}=\RR^d$,
$\mathrm{A}\equiv\{0\}$, $\mathrm{C}_1=\mathrm{C}_2\equiv0$ and, for 
every 
$i\in\{1,\ldots,m\}$, $\mathrm{G}_i=\RR$, $\mathrm{D}_i^{-1}=0$, 
$\mathrm{L}_i=\mathrm{a}_i^{\top}$, and 
$\mathrm{B}_i=\partial\mathrm{f}_i$. By using 
Theorem~\ref{p:primdumi} in this case for matrices 
$(\mathrm{P}_{ij})_{0\leq i<j\leq m}$ given by 
\begin{equation}
(\forall 0\leq j<i\leq m)\quad \mathrm{P}_{ij}=
\begin{cases}
\frac{\Id}{\sigma_0},\quad&\text{if }i=j=0;\\
\frac{1}{\sigma_i},\quad&\text{if }i=j>0;\\
-\mathrm{a}_i^{\top},&\text{if }j=0;\\
\sigma_0 \mathrm{a}_i^{\top}\mathrm{a}_j,&\text{if }0<j<i,
\end{cases}
\end{equation}
we obtain
\begin{equation}
 \label{e:algomicompERM3} 
 \begin{array}{l}
 \left\lfloor
 \begin{array}{l}
 \mathrm{v}_1^k = \prox_{ 
  \sigma_{1}\mathrm{f}_1^{*}}\left(\mathrm{u}_1^k + 
   \sigma_{1}\left(    
   \mathrm{a}_1^{\top}\mathrm{x}^k-\sigma_{0}\sum_{i=1}^m
    \mathrm{a}_1^{\top}\mathrm{a}_i\mathrm{u}_i^k\right)\right)\\
  \mathrm{v}_2^k = \prox_{ 
    \sigma_{2}\mathrm{f}_2^{*}}\left(\mathrm{u}_2^k+
    \sigma_{2}\left(
     \mathrm{a}_2^{\top}\mathrm{x}^k-\sigma_{0}\left(\mathrm{a}_2^{\top}
     \mathrm{a}_1\mathrm{v}_1^k+\sum_{i=2}^m
 \mathrm{a}_2^{\top}\mathrm{a}_i\mathrm{u}_i^k\right)\right)\right)\\[1mm]
     \hspace{0.5cm}\vdots\\
  \mathrm{v}_m^k = \prox_{ 
      \sigma_{m}\mathrm{f}_m^{*}}\left(\mathrm{u}_m^k + 
 \sigma_{m}\left(
  \mathrm{a}_m^{\top}\mathrm{x}^k-\sigma_{0}\left(\sum_{i=1}^{m-1}
      \mathrm{a}_m^{\top}\mathrm{a}_i\mathrm{v}_i^k
  +\|\mathrm{a}_m\|^2\mathrm{u}_m^k\right)\right)\right)\\[1mm]
 \mathrm{x}^{k+1}= 
 \mathrm{x}^k-{\lambda}\sum_{i=1}^m
  \mathrm{a}_i\mathrm{v}_i^k\\
    \mathrm{u}_1^{k+1}=\mathrm{u}_1^k+\frac{\lambda}{\sigma_1}\left(
    \mathrm{v}_1^k-\mathrm{u}_1^k\right)\\
\hspace{0.8cm}     \vdots\\
    \mathrm{u}_m^{k+1}=\mathrm{u}_m^k+\frac{\lambda}{\sigma_m}\left(
    \mathrm{v}_m^k-\mathrm{u}_m^k\right)+
\sigma_0\sum_{j=1}^{m-1}\mathrm{a}_m^{\top}\mathrm{a}_j(\mathrm{v}_j^k
-\mathrm{u}_j^k).        
 \end{array}
 \right.\\[2mm]
 \end{array}
  \end{equation}
  Since conditions \eqref{e:condiPii}-\eqref{e:metricconditionpd}
  hold if 
  \begin{equation}
  \label{e:condERM1}
  \sqrt{\sum_{i=1}^m\|\mathrm{a}_i\|^2}+\sigma_0\sum_{i=1}^m
  \|\mathrm{a}_i\|^2+\frac{\sigma_0}{2}
  \left(\max_{i=1,\ldots,m}\|\mathrm{a}_i\|^2-
  \min_{i=1,\ldots,m}\|\mathrm{a}_i\|^2\right)<
  \frac{1}{\max\limits_{i=0,\ldots,m}\sigma_i},
  \end{equation}
  by choosing $(\sigma_i)_{0\leq i\leq m}$ satisfying 
  \eqref{e:condERM1}
  the sequence $(\mathrm{x}^k)_{k\in\NN}$ generated by 
  \eqref{e:algomicompERM3} converges to a solution provided that
  $\lambda<M^{-1}$ where 
  $$M=\left(\min_{i=0,\ldots,m}\sigma_i\right)^{-1}+\frac{1}{2}
  \sqrt{\sum_{i=1}^m\|\mathrm{a}_i\|^2}+\frac{\sigma_0}{2}\left(\sum_{i=1}^m
  \|\mathrm{a}_i\|^2+\max_{i=1,\ldots,m}\|\mathrm{a}_i\|^2\right).$$
 Note that, without loss of generality, we can assume, for every 
 $i\in\{1,\ldots,m\}$, $\|\mathrm{a}_i\|=1$, since 
 $\mathrm{f}_i(\mathrm{a}_i^{\top}\mathrm{x})=
 \mathrm{g}_i((\mathrm{a}_i/\|\mathrm{a}_i\|)^{\top}\mathrm{x})$
 with $\mathrm{g}_i\colon \mathrm{x}\mapsto 
 \mathrm{f}_i(\|\mathrm{a}_i\|\mathrm{x})$ and 
 $\prox_{\mathrm{g}_i}\colon\mathrm{x}\mapsto 
 \prox_{\|\mathrm{a}_i\|^2\mathrm{f}_i}
 (\|\mathrm{a}_i\|\mathrm{x})/\|\mathrm{a}_i\|$. Therefore, condition
 \eqref{e:condERM1} can be reduced to 
 $\sqrt{m}+m\sigma_0<(\max_{i=0,\ldots,m}\sigma_i)^{-1}$, which, in 
 the 
 case $\sigma_0=\cdots=\sigma_m$ reduces to 
 $\sigma_0<(\sqrt{5}-1)/(2\sqrt{m})$.

{
\subsection{A Distributed Operator Splitting Scheme with Time-Varying Networks}

In this section we develop an extension of the popular distributed operator splitting 
scheme \emph{PG-Extra}~\cite{shi2015extra,shi2015proximal} to time-varying graphs. The 
problem data are a collection of cost functions $f_1, \ldots, f_n$ on a Hilbert space $\cH$ 
and a sequence of connected, undirected communication graphs $ G_t = (V_t, E_t)$ with 
vertices $V_t = \{1, \ldots, n\}$ and edges $ E_t \subseteq \{1, \ldots, n \}^2$. Then the goal 
of distributed optimization is to 
\begin{align}\label{eq:logs}
\Min_{x \in \cH} \; \sum_{i =1}^n f_i(x), 
\end{align}
through an iterative algorithm that, at every time $t \in \NN$, only allows communication between neighbors in $G_t$. For simplicity, we focus on the case wherein 
$f_i : \cH \rightarrow \RX$ is proper, lower semicontinuous and convex.  

A well-known distributed operator splitting schemes is known as \emph{PG-Extra}. This method applies to fixed communicated graphs $G_t \equiv G$, and can be viewed as an instance of modern primal-dual algorithms, such as Condat-Vu~\cite{condat2013primal,vu2013splitting}. To the best of our knowledge there is no known extension of \emph{PG-Extra} to time-varying graphs that may also be applied to monotone inclusions. We will now develop such an extension.

For the graph $G_t$, let $A_t$ denote its adjacency matrix and let $D_t$ denote its degree matrix.\footnote{Briefly, $(A_t)_{ij} = 1$ if $(i,j) \in E$ and is zero otherwise, while $D$ is a diagonal matrix with diagonal entries $D_{ii} = \text{deg}(i)$.} The Laplacian matrix of $G_t$ is defined as the difference
$$
L_t := D_t - A_t.
$$
It is well-known that, for fully connected graphs, we have the identity $\ker(L_t) = \mathrm{span}(\mathbf{1}_n)$~\cite{chung1997spectral}. We may exploit this fact to develop an equivalent formulation of~\eqref{eq:logs}.  

The Laplacian operator has a natural extension to the product space $\cH^n$. It is then a straightforward exercise to show that the extension induces the following identity: 
\begin{align*}
\left(\forall \mathbf{x} := (x_1, \ldots, x_n) \in \cH^n\right)   &&L_t\mathbf{x} = 0 \iff x_1 = x_2 = \cdots = x_n.
\end{align*}
Therefore, a family of equivalent formulations of~\eqref{eq:logs} is given by 
\begin{align*}
\Min_{\mathbf{x} \in \cH^n} &\; \sum_{i \in V} f_i(x_i) \\
\text{subject to:} & \;  L_t \mathbf{x} = 0. \numberthis\label{eq:new_reform}
\end{align*}
The constraint $L_t\mathbf{x} = 0$ is equivalent to the constraint $x \in \cU := \{x \in \cH^n 
\mid x_1 = \ldots = x_n\}$. Thus, one could apply a splitting method to derive a distributed 
algorithm consisting of decoupled proximal steps on the $f_i$ followed by \emph{global 
averaging} steps induced by the projection onto $\cU$. However, in order to develop an 
algorithm that respects the \emph{local communication structure} of the graphs $G_t$, we 
must avoid computing such projections onto $\cU$. For any fixed $t$, we may develop 
such a method as a special case of modern primal-dual algorithms.

Indeed, a straightforward application of Condat-Vu~\cite{condat2013primal,vu2013splitting} yields the update rule
\begin{align*}
\text{For all $i \in V$} & \text{ in parallel} \\
x_i^{k+1} &= \prox_{\gamma f_i}(x_i^{k} -\gamma (L_t\mathbf{y}^k)_i)\\
\mathbf{y}^{k+1} &= \mathbf{y}^{k} + \tau L_t (2\mathbf{x}^{k+1} - \mathbf{x}^k), \numberthis\label{eq:condat_vu_decentralized}
\end{align*}
where $\gamma, \tau > 0$ are appropriately chosen stepsizes. This algorithm is fully decentralized because multiplications by $L_t$ only induce communication among neighbors in the graph $G_t$. 

If we allow $t = k$, this Condat-Vu~\cite{condat2013primal,vu2013splitting} algorithm has, to the best of our knoweldge, no supporting convergence theory, although each of the optimization problems~\eqref{eq:new_reform} have the same set of solutions. The lack of convergence theory arises because Condat-Vu measures convergence in the product space $(\cH^n \times \cH^n, \|\cdot\|_{P_t})$, where $P_t$ is a metric inducing linear transformation depending on $L_t$:
\begin{align*}
P_t  &:= \begin{bmatrix} 
\frac{1}{\gamma} \Id & -L_t \\
-L_t & \frac{1}{\tau} \Id 
\end{bmatrix}.
\end{align*}
One may hope to apply standard variable metric operator-splitting 
schemes~\cite{combettes2012variable,vu2013variableFBF}, but the compatibility condition 
cannot hope to be satisfied. Thus, instead of Condat-Vu, we apply the variable metric 
technique developed in this manuscript.

Mathematically, we let 
\begin{align*}
\cS_t : \cH^n \times \cH^n & \rightarrow \cH^{n} \times \cH^n \\
(\mathbf{x}, \mathbf{y}) &\mapsto ( (\prox_{\gamma f_i}(x_i^{k} - \gamma (L_ty^k)_i))_{i=1}^n , \mathbf{y}^{k} + \tau L_t (2\mathbf{x}^{k+1} - \mathbf{x}^k)).
\end{align*}
Given a proper choice of $\gamma$ and $\tau$, the results of~\cite{condat2013primal,vu2013splitting} show that $\cS_t$ is of $\mathfrak{T}$-class in the space $(\cH^n \times \cH^n, \|\cdot \|_{P_t})$ (indeed, $\cS_t$ is a resolvent). Thus, for any $0 < \mu \leq \|P_t\|^{-1}$, Proposition~\ref{prop:classT} implies that 
$$
\cQ_t = \Id - \mu P_t(\Id - \cS_t), 
$$
is of $\mathfrak{T}$-class in the space $(\cH^n \times \cH^n, \|\cdot \|)$ and $\Fix(\cQ_t) = \Fix(\cS_t)$. Like $\cS_t$, the operator $\cQ_t$ may be computed in a decentralized fashion, as communication between agents is only induced through multiplications by $L_t$.

The algorithm resulting from applying $Q_t$ is a time-varying distributed operator-splitting scheme:
\begin{align*}
(\mathbf{x}^{k+1}, \mathbf{y}^{k+1}) = \cQ_k (\mathbf{x}^{k}, \mathbf{y}^{k}).
\end{align*}
The convergence of this iteration may be proved using an argument similar to Theorem~\ref{cor:asymmetricnoinversion} (which does not capture the case in which the operator at hand is varying). To prove convergence of this iteration, one must observe that the $\Fix(\cQ_k)$ is constant, that for all $(\mathbf{x}^\ast, \mathbf{y}^\ast) \in \Fix(Q_k)$ the sequence $\|((\mathbf{x}^{k}, \mathbf{y}^{k}) - (\mathbf{x}^\ast, \mathbf{y}^\ast)\|$ is nonincreasing, and that $\sum_{k=0}^\infty \| (\mathbf{x}^{k+1}, \mathbf{y}^{k+1}) - (\mathbf{x}^{k}, \mathbf{y}^{k})\|^2 < \infty $. A standard argument then shows that $(\mathbf{x}^{k}, \mathbf{y}^{k})$ converges to an element of $\Fix(\cQ_k) \equiv \Fix(\cQ_0)$.

\subsection{Nonlinear constrained optimization problems}
\label{sec:NLMP}
In this application we aim at solving the nonlinear constrained optimization problem
\begin{equation}
\label{e:nonlMP}
\Min_{x\in C}{f(x)+h(x)},
\end{equation}
where $C=\menge{x\in\HH}{(\forall i\in\{1,\ldots,p\})\quad g_i(x)\le 0}$, $f\colon\HH\to\RX$
is lower semicontinuous, convex and proper, for every $i\in\{1,\ldots,p\}$, 
$g_i\colon\dom(g_i)\subset \HH\to\RR$ and $h\colon\HH\to\RR$ are $\mathcal{C}^1$
convex functions in $\inte\dom g_i$ and $\HH$, respectively, and $\nabla h$ is 
$\beta^{-1}-$Lipschitz. A 
solution of the optimization 
problem \eqref{e:nonlMP}
can be found via the saddle points of the Lagrangian
\begin{equation}
L(x,u)=f(x)+h(x)+u^{\top}{g(x)}-\iota_{\RR^p_+}(u),
\end{equation}
which, under standard qualification conditions can be found by solving the monotone 
inclusion (see \cite{rockafellar1970saddle})
\begin{equation}
{\rm find}\quad x\in Y\quad \text{ such that }\quad (\exi u\in\RR_+^p)\quad  (0,0)\in 
A(x,u)+B_1(x,u)+B_2(x,u),
\end{equation}
where $Y\subset\HH$ is a nonempty closed convex set modeling apriori information on the 
solution (eventually we can take $Y=\HH$), $A\colon (x,u)\mapsto \partial f(x)\times 
N_{\RR^p_+}u$ is maximally 
monotone, $B_1\colon (x,u)\mapsto (\nabla h(x),0)$ is $\beta-$cocoercive, and
$$B_2\colon (x,u)\mapsto \left(\sum_{i=1}^pu_i\nabla g_i(x),-g_1(x),\ldots,-g_p(x)\right)$$ 
is nonlinear, 
monotone and continuous 
\cite{rockafellar1970saddle}. If $Y\subset \dom\partial f\subset\cap_{i=1}^p\inte\dom g_i$
we have that $X:=Y\times\RR_+^p\subset \dom A=\dom\partial f\times \RR_+^p\subset 
\dom B_2=\cap_{i=1}^p\inte\dom g_i\times\RR^p$ and, from 
\cite[Corollary~25.5]{bauschke2017convex},
we have that $A+B_2$ is maximally monotone.
The method proposed in Theorem~\ref{t:1} reduces to
 \begin{equation}
\label{e:algonlc}
(\forall k\in\NN)\quad 
\begin{array}{l}
\left\lfloor
\begin{array}{l}
y^k=\prox_{\gamma_k f}\left(x^k-\gamma_k(\nabla h(x^k)+\sum_{i=1}^pu_i^k\nabla 
g_i(x^k))\right)\\[0.5mm]
\text{For every } i=1,\ldots,p\\[0.5mm]
\left\lfloor
\begin{array}{l}
\eta_i^k=\max\left\{0,u_i^k+\gamma_kg_i(x^k)\right\}\\[0.5mm]
u_i^{k+1}=\max\left\{0,\eta_i^k-\gamma_k(g_i(x^k)-g_i(y^k))\right\}\\[0.5mm]
\end{array}
\right.\\[4mm]
x^{k+1}=P_Y\left(y^k+\gamma_k\sum_{i=1}^p(u_i^k\nabla 
g_i(x^k)-\eta_i^k\nabla 
g_i(y^k))\right),
\end{array}
\right.
\end{array}
\end{equation}
where, for every $k\in\NN$, $\gamma_k$ is found by the backtracking procedure 
defined in \eqref{e:armijocond}. Note that, since $B_2$ is nonlinear, the approaches 
proposed in 
\cite{condat2013primal,vu2013splitting} cannot be applied to this instance.

In the particular instance when $f=\iota_{\Omega}$ for 
some
nonempty closed convex set $\Omega$, we can choose, among other options,  
$Y=\Omega$ since we know 
that any solution must belong to $\Omega$. 
Moroever, when, for every $i\in\{1,\ldots,p\}$, $g_i\colon x\mapsto d_i^{\top}x$, where
$d_i\in\RR^N$, we have $B_2\colon (x,u)\mapsto (D^{\top}u,-Dx)$, where 
$D=[d_1,\ldots,d_p]^{\top}$. This is a particular instance of problem 
\eqref{eq:primal_before_PD} and $B_2$ is $\|D\|-$Lipschitz in this case, which allows us to 
use constant stepsizes $\gamma_k=\gamma\in]0,\chi[$, where $\chi$ is defined in 
\eqref{e:chi} and $L=\|D\|$. 
Theorem~\ref{t:1} guarantees the convergence of the 
iterates $\{x^k\}_{k\in\NN}$ thus generated to a solution to \eqref{e:nonlMP} in any case.

In the next section, we explore some numerical results showing the good performance 
of this method and the method with constant step-size when $g_i$ are affine linear.
}
{
\section{Numerical simulations}
\label{sec:7}
In this section we provide two instances of Section~\ref{sec:NLMP} and we compare 
our proposed method with available algorithms in the literature.

\subsection{Optimization with linear inequalities}
In the context of problem~\eqref{e:nonlMP}, suppose that $\HH=\RR^N$, $h\colon 
x\mapsto \|Ax-b\|^2/2$, $A$ is a $m\times N$ real matrix with $N=2m$ and $b\in\RR^m$, 
$f=\iota_{[0,1]^N}$, and 
$$(\forall i\in\{1,\ldots,p\})\quad g_i(x)=d_i^{\top}x,$$
where $d_1,\ldots,d_p\in\RR^N$. In this case, $B_1\colon (x,u)\mapsto (A^{\top}(Ax-b),0)$,
$B_2\colon (x,u)\mapsto (D^{\top}u,-Dx)$, where $D=[d_1,\ldots, d_p]^{\top}$,
$\beta=\|A\|^{-2}$ and $L=\|D\|$.
We compare the 
method proposed in \eqref{e:algonlc} using the line search (FBHF-LS),
the version with constant stepsize (FBHF), the method proposed by
Condat and V\~u \cite{condat2013primal,vu2013splitting} (CV), the method proposed by
Tseng \cite{tseng2000modified} with line search (Tseng-LS) and with constant stepsize 
(Tseng) for randomly generated matrices and vectors $A$, $D$ and $b$.
We choose the same starting point for each method 
and the parameters for the line search for Tseng-LS and FBHF-LS are $\theta=0.316$, 
$\varepsilon=0.88$ and $\sigma=0.9$. For the constant stepsizes versions of Tseng and 
FBHF, we use $\gamma=\delta/(\beta^{-1}+L)$ and 
$\gamma=\delta\beta/(1+\sqrt{1+16\beta^2L^2})$, respectively, and for $\bar{\sigma}>0$ we 
select $\tau=1/(1/2\beta+\bar{\sigma} L^2)$
in order to satisfy 
the convergence conditions on the parameters of each algorithm. 
We choose several values of $\delta$ and $\bar{\sigma}$ for studying the behavior and
we use the stopping criterion  
$\|(x_{k+1}-x_k,u_{k+1}-u_k)\|/\|(x_k,u_k)\|<10^{-7}$.
In Table~\ref{tab:0} we show the performance of the five algorithms
for random matrices $A$ and $D$ and a random vector $b$ with $N=2000$ and 
$p=100$ and a selection of the parameters $\sigma,\delta$. We see that for Tseng and 
FBHF
the performance improve for larger choices of $\delta$, while for CV it is not clear how to 
choose $\bar{\sigma}$ in general. Even if the theoretical bound of FBHF does not permit 
$\delta$ to go beyond $4$, for $\delta=4.4$ the convergence is also obtained for this case 
with a better performance.
We suspect that the particular structure of this particular case can be exploited for 
obtaining a 
better bound. We also observe that the best performance in time is obtained for the lowest 
number of iterations for each method. In addition, for this instance, algorithms 
FBHF, CV and FBHF-LS are comparable in computational time, while the algorithms by 
Tseng \cite{tseng2000modified} are considerably less efficient in time and in number of 
iterations. In Table~\ref{tab:01} we compare the average time and iterations that the more 
efficient methods in the first simulation take to achieve the stop criterion 
($\epsilon=10^{-7}$) for $20$ random realizations of matrices $A$ and $D$ and a random 
vector $b$, with $N=2000$ and 
$p=100$. We use the parameters yielding the best performance of each method in the first 
simulation. For FBHF we also explore the case when 
$\delta=4.7$, which gives the best performance. We also observe that FBHF for 
$\delta=3.999$ is comparable with CV in average time, while FBHF-LS is slower for this 
instance.
 
\begin{table}[]
	\centering
	\resizebox{\textwidth}{!}{%
		\begin{tabular}{|l||l|l|l||l|l|l||l|l|l|l||l||l|}
			\hline
			$\epsilon=10^{-7}$         & \multicolumn{3}{c||}{Tseng}  & 
			\multicolumn{3}{c||}{FBHF}   
			& \multicolumn{4}{c||}{CV}               & Tseng-LS & FBHF-LS \\ \hline
			$\delta$,$\bar{\sigma}$ & 0.8     & 0.9     & 0.99       & 3.2     & 3.99     & 4.4     & 
			0.0125  & 
			0.0031  & 0.0008  & 0.0002  & LS       & LS      \\ \hline
			$h(x^*)$           & 158.685 & 158.684 & 158.684 & 158.681 & 158.680 & 158.679 & 
			158.674 & 
			158.674 & 158.676 & 158.687 & 158.683  & 158.680 \\ \hline
			iter.       & 20564   & 18482   & 16791   & 11006   & 8915   & 8243    & 9384    & 
			9158    & 8516    & 13375   & 14442    & 10068   \\ \hline
			time (s)    & 41.55   & 37.18   & 33.76   & 13.11   & 10.48   & 9.76    & 10.70   & 10.61   
			& 9.67    & 15.27   & 94.86    & 12.40   \\ \hline
		\end{tabular}%
	}
\vspace{0.2cm}
	\caption{Comparison of Tseng, FBHF and CV (with different values of $\delta$ and 
	$\bar{\sigma}$),
	Tseng-LS and FBHF-LS for a stop criterion of $\epsilon=10^{-7}$.}
	\label{tab:0}
\end{table}

\begin{table}[]
	\centering
	\begin{tabular}{|l|l|l|}
		\hline
		$\epsilon=10^{-7}$ & av. iter.   & av. time (s) \\ \hline
		FBHF-LS          & 36225 & 43.96  \\ \hline
		FBHF ($\delta=3.999$) & 32563 & 39.07 \\ \hline
		FBHF ($\delta=4.7$)   & 28364 & 34.14  \\ \hline
		CV ($\sigma=0.0008$)  & 33308 & 38.60  \\ \hline
	\end{tabular}
	\caption{Average performance of the more efficient methods for $20$ random 
	realizations of $A$, $D$ and $b$ with $N=2000$ and $p=100$.}
	\label{tab:01}
\end{table}

\subsection{Entropy constrained optimization}
\label{sec:71}
In the context of problem~\eqref{e:nonlMP}, suppose that $\HH=\RR^N$, $h\colon 
x\mapsto x^{\top}Qx-d^{\top}x+c$, $Q$ is a $N\times N$ semidefinite positive real 
matrix, $b\in\RR^N$, $c\in\RR$, $f=\iota_{\Omega}$,
$\Omega$ is a closed convex subset of $\RR^N$,
$p=1$, and $$g_1\colon\RR^N_+\to\RR\colon  x\mapsto 
\sum_{i=1}^Nx_i\left(\ln\left(\frac{x_i}{a_i}\right)-1\right)-r,
$$
where $-\sum_{i=1}^Na_i<r<0$, $a\in\RR^N_{++}$ and we use the convention $0\ln(0)=0$. 
This problem appears 
in robust least squares estimation when a relative entropy constraint is included 
\cite{levy04}.
This constraint can be seen as a distance constraint with respect to the vector $a$, where
the distance is measured by the Kullback-Leibler divergence \cite{basseville13}.

In our numerical experience, we assume $Q=A^{\top}A$, $d=A^{\top}b$ and $c=\|b\|^2/2$,
where $A$ is a $m\times N$ real matrix with $N=2m$ and $b\in\RR^m$, which yields 
$h\colon x\mapsto \|Ax-b\|^2/2$, $\beta=\|A\|^{-2}$, $\Omega=[0.001,1]^N$, and 
$a=(1,\ldots,1)^{\top}$.
In this context, $g_1$ achieves its minimum in $\bar{x}=(1,\ldots,1)^{\top}$ and 
$g_1(\bar{x})=-N$
and we choose $r\in]-N,0[$.
Since the constraint is not linear, we cannot 
use the methods proposed in \cite{vu2013splitting,condat2013primal}. We compare the 
method proposed in \eqref{e:algonlc} with line search (FBHF-LS) with the Tseng's method 
with linesearch 
\cite{tseng2000modified} (Tseng-LS) and two routines in matlab: {\tt fmincon.interior-point} 
(FIP) and 
{\tt fmincon.sqp} (SQP). For $m=100, 200, 300$, we generate $20$ random matrices $A$
and random vectors $b$ and we compare the previous methods by changing 
$r\in\{-0.2N,-0.4N,-0.6N,-0.8N\}$
in order to vary the feasible regions. We choose the same starting point for each method 
and the parameters for the line search for Tseng-LS and FBHF-LS are $\theta=0.707$, 
$\varepsilon=0.88$ and $\sigma=0.9$.
The stopping criterion is 
$\|(x_{k+1}-x_k,u_{k+1}-u_k)\|/\|(x_k,u_k)\|<\epsilon$ with $\epsilon=10^{-11}$. In 
Table~\ref{tab:1} we show,
for $m=300$,
the value of the objective function $h$, the nonlinear constraint $g_1$ and time for 
achieving the stopping criterion for a fixed random matrix $A$ and vector $b$
by moving $r\in\{-0.2N,-0.4N,-0.6N,-0.8N\}$. We observe that all methods achieve
almost the same value of the objective function and satisfy the constraints, but
in time FBHF-LS obtains the best performance, even if the number of iterations are larger 
than
that of FIP and SQP. Tseng-LS has also a better performance in time than FIP and SQP, with 
a much larger number of iterations. We also observe that, the smaller the feasible 
set is, the harder is for all the methods to approximate the solution and the only case 
when the constraint is inactive is when $r=-0.2N$. On the other hand, even if in the cases 
$r=-0.6N$ and $r=-0.8N$ we have $g_1(x^*)>0$, the value is $\approx 10^{-6}$ which is 
very near to feasibility.
This behavior is confirmed in Table~\ref{tab:2}, in which we show,
for each $m\in\{100,200,300\}$, the average time and iterations 
obtained from the $20$ random realizations by moving 
$r\in\{-0.2N,-0.4N,-0.6N,-0.8N\}$. We observe that FBHF-LS takes considerably less time
than the other algorithms to reach the stopping criterion and the difference is more when 
dimension is higher. 
Since FIP and SQP are very slow for high dimensions, in Table~\ref{tab:3}
we compare the efficiency of Tseng-LS and FBHF-LS for $20$ random realizations
of $A$ and $b$ with $N\in\{1000,2000,3000\}$ for $r=-0.4N$ and $\epsilon=10^{-5}$.
The computational time of both methods are reasonable, but again FBHF-LS is faster.
FBHF-LS use less iterations than Tseng-LS for achieving the same criterion and, even if
we reduce $\epsilon$ from $10^{-5}$ to $10^{-10}$ and the number of iterations are more 
than 3 times that of Tseng-LS for the weaker criterion, the computational time is similar.
We also observe that the percentage of relative improvement of an algorithm $A$ with 
respect to Tseng-LS,
measured via $\% imp.(A)=100*(f(x_A)-f(x_T))/f(x_T)$, where $x_T$ and $x_A$ are the 
approximative solutions obtained by Tseng-LS and $A$, is bigger for smaller dimensions.
For instance, in the case $500\times 1000$, FBHF-LS obtain an approximative solution 
for which the objective function has a $12\%$ of relative improvement with respect to that 
of 
Tseng-LS for $\epsilon=10^{-5}$ and, if the criterion is strengthened to $10^{-10}$, the 
improvement raises to $20\%$. For higher dimensions, this quantities are considerably 
reduced.

\begin{table}[]
	\centering
	\resizebox{\textwidth}{!}{%
		\begin{tabular}{|l||l|l|l|l||l|l|l|l||l|l|l|l||l|l|l|l|}
			\hline
			& \multicolumn{4}{c||}{$r=-0.2N$}           & 
			\multicolumn{4}{c||}{$r=-0.4N$}           & \multicolumn{4}{c||}{$r=-0.6N$}           & 
			\multicolumn{4}{c|}{$r=-0.8N$}           \\ \hline
			
			$300\times600$ & $h(x^*)$   & $g_1(x^*)$ & time (s) &iter. & $h(x^*)$  & $g_1(x^*)$ 
			& 
			time (s) & iter.& $h(x^*)$  & $g_1(x^*)$ & time (s) &iter.& $h(x^*)$ & $g_1(x^*)$ & 
			time (s)&iter. \\ \hline
			FIP          & 6.06E-10  & -105.038& 165.190  & 562  & 3.73E-09  & -7.78E-02    & 
			183.467  & 574   & 244.551 & -3.39E-08    & 218.413  & 794   & 4075.6824 & 
			-3.20E-09    & 378.269    & 1197  \\
			SQP          & 6.06E-10  & -105.038 & 372.210  & 357  & 3.73E-09  & -7.78E-02    & 
			598.258  & 341   & 244.551 & -3.39E-08    & 515.568  & 653   & 4075.6824 & 
			-3.20E-09    & 988.143    & 655   \\
			Tseng-LS     & 5.47E-15  & -119.365 & 13.682   & 13785 & 1.41E-14  & -9.81E-09    & 
			29.160   & 30110 & 244.551 & 1.16E-06     & 44.248   & 20717 & 4075.6822  & 
			2.67E-06     & 110.916    & 75254 \\
			FBHF-LS        & 2.15E-15  & -119.338 & 1.053    & 9680  & 5.56E-15  & -4.71E-09    & 
			2.220    & 21106 & 244.551 & 1.10E-06     & 10.381   & 19492 & 4075.6822 & 
			2.07E-06     & 17.464     & 60442\\
			 \hline
		\end{tabular}%
	}
	\vspace{.2cm}
	\caption{Comparison of objective function and constraints values, time and number of 
	iterations of FIP, SQP, Tseng-LS and FBHF-LS algorithms for solving the entropy 
	constrained 
		optimization when $N=600$, $m=300$ and $r\in 
		\{-0.2N,-0.4N,-0.6N,-0.8N\}$.}
		\label{tab:1}
\end{table}

\begin{table}[]
	\centering
	\resizebox{\textwidth}{!}{%
		\begin{tabular}{|l||l|l|l|l||l|l|l|l||l|l|l|l|}
			\hline
			Time (s) & \multicolumn{4}{c||}{$100\times 200$}          & 
			\multicolumn{4}{c||}{$200\times 
			400$}          
			& 
			\multicolumn{4}{c|}{$300\times 600$}          \\ \hline
			constraint & $r=-0.2N$ & $r=-0.4N$ & $r=-0.6N$ & $r=-0.8N$ & $r=-0.2N$ & 
			$r=-0.4N$ & $r=-0.6N$ 
			& $r=-0.8N$ & $r=-0.2N$ & $r=-0.4N$ & $r=-0.6N$ & $r=-0.8N$ \\ \hline
			FIP        & 8.24    & 9.50    & 11.92   & 11.22   & 52.95   & 57.92   & 75.02   & 76.29   & 
			142.51  & 183.22  & 253.80  & 324.42  \\ 
			SQP        & 8.60    & 11.18   & 14.28   & 18.51   & 70.88   & 98.70   & 122.03  & 209.73  
			& 313.52  & 489.95  & 569.34  & 1075.42 \\
			Tseng-LS   & 22.92   & 10.71   & 7.92    & 9.91    & 81.16   & 13.46   & 39.17   & 83.01   
			& 139.47  & 26.50   & 84.07   & 111.44  \\ 
			FBHF-LS    & 2.21    & 0.99    & 2.72    & 2.51    & 7.06    & 0.95    & 10.23   & 18.09   
			& 12.48   & 1.88    & 20.59   & 18.30   \\ \hline
		\end{tabular}%
	}
	\vspace{.2cm}
	\caption{Average time (s) to reach the stopping criterion of $20$ random realizations for 
	FIP, 
	SQP, Tseng-LS and FBHF-LS for a matrix $A$ with
	dimension $100\times 200$, $200\times 400$ and $300\times 600$ and $r\in 
	\{-0.2N,-0.4N,-0.6N,-0.8N\}$.}
	\label{tab:2}
\end{table}

\begin{table}[]
	\centering
	\resizebox{\textwidth}{!}{%
	\begin{tabular}{|l|l||l|l|l||l|l|l||l|l|l|}
		\hline
		&           & \multicolumn{3}{c||}{$500\times1000$}  & 
		\multicolumn{3}{c||}{$1000\times2000$} & \multicolumn{3}{c|}{$1500\times3000$} \\ 
		\hline
		stop crit.                                                                & Algorithm & time (s) & iter. & \% 
		imp. & time (s) & iter. & \% imp. & time (s) & iter. & \% imp. \\ \hline
		\multicolumn{1}{|c|}{\multirow{2}{*}{$10^{-5}$}} & Tseng-LS  & 17.22    & 2704  & 
		0       & 52.80    & 4027  & 0       & 91.23    & 3349  & 0      \\ 
		\multicolumn{1}{|c|}{}                                                    & FBHF-LS   & 2.33     & 1993  & 
		12.2       & 6.26     & 3239  & 3.2       & 9.56     & 2474  & 0.1       \\ \hline
		\multicolumn{1}{|c|}{$10^{-10}$}                                     & FBHF-LS   & 10.67    & 10092 
		& 
		20.1       & 71.51    
		& 
		33637 & 6.5       & 53.47    & 12481 & 0.2       \\ \hline
	\end{tabular}%
}
	\vspace{.2cm}
\caption{Comparison between Tseng-LS and FBHF-LS for higher dimensions. We compare 
average time and average number of iterations for achieving stop criteria together with 
percentage of relative improvement with respect to Tseng-LS approximate solution.}
\label{tab:3}
\end{table}
}

\section{Conclusion}
In this paper, we systematically investigated a new extension of Tseng's 
forward-backward-forward method and the forward-backward method. The three primary 
contributions of this investigation are (1) a lower per-iteration complexity variant of 
Tseng's method which activates the cocoercive operator only once; (2) the ability to 
incorporate variable metrics in operator-splitting 
schemes, which, unlike typical variable metric methods, do not enforce 
compatibility conditions between metrics employed at successive time steps; and (3) 
the ability to incorporate modified resolvents $J_{P^{-1} A}$ in iterative fixed-point 
algorithms, which, unlike typical preconditioned fixed point iterations, can be formed 
from non self-adjoint linear operators $P$, which lead to new 
Gauss-Seidel style operator-splitting schemes.
\vspace{.3cm}

{\footnotesize {\bf Acknowledgments:} This work is partially supported by NSF GRFP grant 
	DGE-0707424, 
	by CONICYT grant FONDECYT 11140360, and by ``Programa de financiamiento 
	basal'' from CMM, Universidad de Chile. We want to thank the two anonymous reviewers, 
	whose comments and concerns allowed us to improve the quality of this manuscript.}

\bibliographystyle{siam}
\bibliography{bibliographyalm}

\end{document}